\documentclass[oneside,reqno]{amsart}
\usepackage[utf8]{inputenc}
\usepackage{amsmath}
\usepackage{amsfonts}
\usepackage{amssymb}
\usepackage{setspace}
\usepackage{amssymb}
\usepackage[shortlabels]{enumitem}
\usepackage{xcolor}
\usepackage{graphicx}
\usepackage{amsthm}
\newtheorem{thm}{Theorem}[section]
\newtheorem{df}{Definition}[section]

\numberwithin{equation}{section}
\usepackage{bbm}
\newtheorem{rmk}{Remark}[section]

\newtheorem{prop}{Proposition}[section]
\newtheorem{lm}{Lemma}[section]

\usepackage{geometry}
\usepackage{fancyhdr}
\usepackage{hyperref}
\hypersetup{
	colorlinks=true,
	linkcolor=blue,
	filecolor=magenta,      
	urlcolor=blue,
	pdftitle={Overleaf Example},
	citecolor=red,
	pdfpagemode=FullScreen,
}
\date{}
\begin{document}
	\title[Multi-objective control for stochastic parabolic equations]{Multi-objective control for stochastic parabolic equations with dynamic boundary conditions}
	\author{Omar Oukdach\,$^1$, Said Boulite\,$^2$,  Abdellatif Elgrou\,$^3$ and  Lahcen Maniar\,$^3$}
	\begin{abstract}
		This paper deals with a  hierarchical multi-objective  control problem for forward stochastic parabolic equations with dynamic boundary conditions. The controls are divided into two classes:  leaders and  followers. The goal of the leaders  is of null controllability type while the followers are in charge of  letting the state close to prescribed targets in  fixed observation regions. To solve the problem,  Nash and Stackelberg strategies are used.  To implement these strategies, we combine some appropriate Carleman estimates and the well-known control duality approach. 
	\end{abstract}
	
	\dedicatory{\large Dedicated to the memory of Professor Hammadi Bouslous}
	\maketitle
	
	\footnotetext[1]{\author{Moulay Ismaïl University of Meknes, FST Errachidia,  MSISI Laboratory, AM2CSI Group, BP
			509, Boutalamine, Errachidia, Morocco}. E-mail: \href{omar.oukdach@gmail.com}{\texttt{omar.oukdach@gmail.com}}}
	\footnotetext[2]{Cadi Ayyad University, National School of Applied Sciences, LMDP, UMMISCO (IRD-UPMC), Marrakesh P.B. 575, Morocco. E-mail: \href{s.boulite@uca.ma}{\texttt{s.boulite@uca.ma}}}
	\footnotetext[3]{Cadi Ayyad University, Faculty of Sciences Semlalia, LMDP, UMMISCO (IRD-UPMC), P.B. 2390, Marrakesh, Morocco. E-mail: \href{abdellatif.elgrou@ced.uca.ma}{\texttt{abdellatif.elgrou@ced.uca.ma}}, \href{maniar@uca.ma}{\texttt{maniar@uca.ma}}}

	\section{Introduction}
	A multi-objective optimal control problem is an optimal control problem that involve more than one objective function. In general, the objectives are in conflict, and the optimality of some objectives usually do not lead to the optimality of others, and it is impossible to find a strategy satisfying all the objectives. Then, some trade-offs and compromise are needed to achieve a certain balance of objectives. Many applications from industry, economy and engineering lead to multi-objective optimal controls. For example, in the heat transfer, in a room, it is meaningful to try to guide the temperature to be close as much as possible to a fixed target at the end of the day and, additionally, keep the temperature not too far from a prescribed value at some regions. This can be done by applying several controls at different locations of the room. This problem can be seen as a game with controls as players.
	
	In contrast with the single-objective case, we can give several different definitions of optimal solution depending on the characteristics of the considered problem. In general, it is not possible to find a single solution that would be optimal for all the objectives simultaneously. The concept of strategies or equilibriums are then used as solution notions. The well-known strategies in the literature include the Pareto strategy, the Nash strategy and the Stackelberg strategy, see respectively  \cite{ Na51, Pa96, St34} and references therein.
	
	The aim of this paper is to apply the Stackelberg and Nash strategies for  stochastic parabolic systems with dynamic boundary conditions. Let us describe briefly the above strategies in an abstract setting. Let  $H_1$ and $H_2$  be two Hilbert spaces, and $J_{i}:  H_1\times H_2\longrightarrow \mathbb{R},$ $i=1, 2$,  be  two cost functions. Let us start by describing the Stackelberg's strategy. Suppose we have a hierarchical game with two players: the leader and the follower, with associated cost functions $J_1$ and $J_2$, respectively. The Stackelberg solution to the game is achieved when the follower is forced to wait until the leader announces his policy, before making his own decision. More precisely, we have the following definition, see for instance \cite{StDf}.
	\begin{df}\label{defst}
		A vector $(v_1^{\star}, v_2^{\star})\in H_1\times H_2$  is a Stackelberg strategy  for $J_1$ and $J_2$   if 
		\begin{enumerate}
			\item[(i)]  There is a function $F: H_1\longrightarrow H_2$ such that
			$$J_2(v_1,F(v_1))= \inf\limits_{v\in H_2}J_2(v_1,v), \qquad \forall v_1\in H_1.$$
			\item [(ii)]The control $v^{\star}_1\in H_1$ satisfies 
			$$J_1(v^{\star}_1,F(v^{\star}_1))= \inf\limits_{v\in H_1}J_1(v,F(v)).$$
			\item [(iii)]The pair $(v_1^{\star}, v_2^{\star})$ satisfies $v_2^{\star}= F(v^{\star}_1).$
		\end{enumerate}
	\end{df}
	\noindent Now, we define the Nash strategy. Roughly speaking, it consists of a combination of strategies such that no player improves his gain by changing his strategy, while the other players do not change their strategy. Mathematically, we have the following definition. 
	\begin{df}
		A vector $v^{\star}=(v^{\star}_1, v^{\star}_2)\in H_1\times H_2$ is   a Nash equilibrium for $(J_1, J_2)$ if 
		\begin{equation*}
			\left\{\begin{array}{ll}
				J_1(v^{\star}_1, v^{\star}_2)= \inf\limits_{v\in H_1} J_1(v, v^{\star}_2), \\
				J_2(v^{\star}_1, v^{\star}_2)= \inf\limits_{v\in H_2} J_2(v^{\star}_1, v).
			\end{array}\right.
		\end{equation*}
	\end{df}
	\noindent
	By the well-known results of convex analysis and optimization, we have the following characterization of the Nash equilibrium.
	\begin{rmk}
		Assume that $J_i:  H_1\times H_2 \longrightarrow \mathbb{R}$, $i=1, 2$, are differentiable and convex.  
		A  vector  $(v^{\star}_1,v^{\star}_2)\in H_1\times H_2$ is   a Nash equilibrium for $(J_1, J_2)$ if and only if 
		\begin{equation*}
			\dfrac{\partial J_i}{\partial v_i}(v^{\star}_1,v^{\star}_2)=0, \quad \text{for all} \;\,\,i=1, 2.
		\end{equation*}	
	\end{rmk}
	To present the Nash-Stackelberg equilibrium, suppose that there is a player 3 with a cost function $J: H\times H_1\times H_2 \longrightarrow \mathbb{R}$, where $H$ is a third Hilbert space. We assume that player 3 is the leader, and players 1 and 2 with cost functions $J_1$ and $J_2$, respectively, are the followers. Here, $J_1$ and $J_2$ are defined on $H\times H_1\times H_2$. The Stackelberg-Nash equilibrium is defined as follows.
	\begin{df}
		A vector $(v^{\star}, v_1^{\star}, v_2^{\star})\in H\times H_1\times H_2$ is a Stackelberg-Nash equilibrium  for $(J,J_1,J_2)$ with player 3 as  leader and   players 1 and 2 as followers  if 
		\begin{itemize}
			\item[(i)] There exists two functions  $F_1:  H\longrightarrow H_1  $ and $F_2: H\longrightarrow H_2$ such that
			\begin{align*}
				\begin{cases}
					J_1(v, F_1v,F_2v)&= \inf\limits_{w\in H_1}J_1( v,w,F_2v), \qquad \forall v\in H,\\
					J_2(v, F_1v,F_2v)&= \inf\limits_{w\in H_2}J_2( v,F_1v,w), \qquad \forall v\in H.
				\end{cases}
			\end{align*}
			\item [(ii)]  $v^{\star}\in H$ satisfies 
			$$J(v^{\star},F_1(v^{\star}),F_2(v^{\star} ))= \inf\limits_{v\in H}J(v,F_1v,F_2v).$$
			\item [(iii)]  $(v^{\star}, v^{\star}_1,v^{\star}_2)$ satisfies $v^{\star}_1=F_1(v^{\star})$ and $v^{\star}_2=F_2(v^{\star}).$
		\end{itemize} 
	\end{df}
	
	In the last years, these strategies have been extensively analyzed in the context of deterministic partial differential equations (PDEs for short). We refer to \cite{ArFeGu17, ArCaSa15, Calsavara, GRP02, LiPa, HSP18} for some works on Stackelberg-Nash controllability problems. In these papers, the authors consider the classical boundary conditions: Dirichlet, Neumann, or Robin. For PDEs with dynamic boundary conditions, the literature is also rich. See, for instance, \cite{ACMO20, OuBoMa21, BCMO20, BFurIman, KMO22, KhMa19, Lionsopt71, maniar2017null, zabczyk2020mathematical} for some classical control problems and \cite{BoMaOuNash,BoMaOuNash2} for some control issues using concepts from game theory. In the stochastic case, there are many works on controllability for stochastic partial differential equations, see, e.g., \cite{barbu2003carleman, fadili, liu2014global, lu2011some, lu2021mathematical, san23, tang2009null, yan2018carleman, observineqback, StNash3}. For recent works concerning controllability of stochastic parabolic equations with dynamic boundary conditions, we refer to \cite{Preprintelgrou23, elgrouDBC, BackSPEwithDBC, elgmanforwione}. For some works on stochastic Stackelberg-Nash controllability issues, see, for instance, \cite{StNash4, StNash3, StNash2} and references therein. See also \cite{StNash}, where the authors consider stochastic degenerate parabolic equations.
	
	In this paper, we combine techniques from control theory and game theory to address stochastic parabolic equations with dynamic boundary conditions. Specifically, we utilize a hierarchy of distributed controls as follows: with $v_1$ and $v_2$ (the followers), we aim to maintain the state close to fixed targets within a specified observation region; subsequently, using three controls $u_1$, $u_2$, and $u_3$ (the leaders), we aim to precisely drive the state to rest. The resolution strategy can be summarized as follows: (i) we establish the existence and uniqueness of the Nash equilibrium; (ii) we characterize the Nash equilibrium through an appropriate backward system; (iii) by a Carleman estimate for a suitable forward-backward coupled system, we prove the controllability result as announced.
	
	It is worth noting that the introduction of three leaders is a purely technical constraint specific to forward stochastic parabolic systems. We will see that the extra controls $u_2$ and $u_3$ are essential for attaining our desired controllability result. While considering only one leader $u_1$ would be more realistic, currently, it presents a difficult and challenging problem. Further explanations on this matter can be found in the discussion at the end of the paper.\\
	
	The plan for the rest of the paper is as follows: In Section \ref{section2}, we formulate the problem under consideration. Section \ref{section3} presents the well-posedness of the equations. Section \ref{section4} deals with the existence, uniqueness, and characterization of the Nash equilibrium. In Section \ref{section5}, we prove the needed Carleman estimates. Section \ref{section6} is devoted to establishing our controllability result. We conclude the paper with some discussions in Section \ref{section7}.
	\section{Problem formulation and the main result}\label{section2}
	Let $T>0$, $G\subset\mathbb{R}^N$ (with $N\geq2$) be a nonempty open bounded domain with a smooth boundary $\Gamma=\partial G$, and let the sets $G_0$, $G_1$, $G_2$, $G_{1,d}$, and $G_{2,d}$ be nonempty open subsets of $G$. We also indicate by $\chi_{G_i}$ ($i=0,1,2$) the characteristic function of $G_i$. Put
	$$Q=(0,T)\times G, \,\,\quad \Sigma=(0,T)\times\Gamma,\,\,\quad  \text{and}\,\, \quad Q_0=(0,T)\times G_0.$$
	
	Let $(\Omega,\mathcal{F},\{\mathcal{F}_t\}_{t\in[0,T]},\mathbb{P})$ be a fixed complete filtered probability space on which a one-dimensional standard Brownian motion $W(\cdot)$ is defined such that $\{\mathcal{F}_t\}_{t\in[0,T]}$ is the natural filtration generated by $W(\cdot)$ and augmented by all the $\mathbb{P}$-null sets in $\mathcal{F}$. For a Banach space $\mathcal{X}$, we denote by $C([0,T];\mathcal{X})$ the Banach space of all $\mathcal{X}$-valued continuous functions defined on $[0,T]$; and $L^2_{\mathcal{F}_t}(\Omega;\mathcal{X})$ denotes the Banach space of all $\mathcal{X}$-valued $\mathcal{F}_t$-measurable random variables $X$ such that $\mathbb{E}\big(\vert X\vert_\mathcal{X}^2\big)<\infty$, with the canonical norm; and $L^2_\mathcal{F}(0,T;\mathcal{X})$ indicates the Banach space consisting of all $\mathcal{X}$-valued $\{\mathcal{F}_t\}_{t\in[0,T]}$-adapted processes $X(\cdot)$ such that $\mathbb{E}\big(\vert X(\cdot)\vert^2_{L^2(0,T;\mathcal{X})}\big)<\infty$, with the canonical norm; and $L^\infty_\mathcal{F}(0,T;\mathcal{X})$ is the Banach space consisting of all $\mathcal{X}$-valued $\{\mathcal{F}_t\}_{t\in[0,T]}$-adapted essentially bounded processes, with its norm $|\cdot|_\infty$; and $L^2_\mathcal{F}(\Omega;C([0,T];\mathcal{X}))$ defines the Banach space consisting of all $\mathcal{X}$-valued $\{\mathcal{F}_t\}_{t\in[0,T]}$-adapted continuous processes $X(\cdot)$ such that $\mathbb{E}\big(\vert X(\cdot)\vert^2_{C([0,T];\mathcal{X})}\big)<\infty$, with the canonical norm. We now introduce the following space
	$$\mathbb{L}^2:=L^2(G,dx)\times L^2(\Gamma,d\sigma),$$
	which $dx$ denoted the Lebesgue measure in $G$ and $d\sigma$ indicates the surface measure on $\Gamma$. Equipped with the following standard inner product 
	$$\langle(y,y_\Gamma),(z,z_\Gamma)\rangle_{\mathbb{L}^2} = \langle  y,z\rangle_{L^2(G)} + \langle y_\Gamma,z_\Gamma\rangle_{L^2(\Gamma)},$$
	the space $\mathbb{L}^2$ is a Hilbert space. Moreover, we adopt the following notations
	$$\mathcal{U}=L^2_\mathcal{F}(0,T;L^2(G_0))\times L^2_\mathcal{F}(0,T;L^2(G))\times L^2_\mathcal{F}(0,T;L^2(\Gamma)),$$
	and
	$$\mathcal{V}_i=L^2_\mathcal{F}(0,T;L^2(G_i)),\quad \mathcal{V}_{i,d}=L^2_\mathcal{F}(0,T;L^2(G_{i,d})),\qquad\textnormal{for}\;\;\;i=1,2.$$
	
	The main purpose of this paper is to study the Stackelberg-Nash null controllability for the following stochastic parabolic equation with dynamic boundary conditions
	\begin{equation}\label{eqq1.1}
		\begin{cases}
			\begin{array}{ll}
				dy - \Delta y \,dt = [a_1y+\chi_{G_0}(x)u_1+\chi_{G_1}(x)v_1+\chi_{G_2}(x)v_2] \,dt + (a_2y +u_2)\,dW(t)&\textnormal{in}\,\,Q,\\
				dy_\Gamma-\Delta_\Gamma y_\Gamma \,dt+\partial_\nu y \,dt = b_1y_\Gamma \,dt+(b_2y_\Gamma  +\,u_3) \,dW(t) &\textnormal{on}\,\,\Sigma,\\
				y_\Gamma(t,x)=y\vert_\Gamma(t,x) &\textnormal{on}\,\,\Sigma,\\
				(y,y_\Gamma)\vert_{t=0}=(y_0,y_{\Gamma,0}) &\textnormal{in}\,\,G\times\Gamma,
			\end{array}
		\end{cases}
	\end{equation}
	where $a_1, a_2\in L_\mathcal{F}^\infty(0,T;L^\infty(G))$, $b_1, b_2\in L_\mathcal{F}^\infty(0,T;L^\infty(\Gamma))$, $(y_0,y_{\Gamma,0})\in L^2_{\mathcal{F}_0}(\Omega;\mathbb{L}^2)$ is the initial state and $(y,y_\Gamma)$ is the state variable. The leaders are constituted of
	$(u_1,u_2,u_3) \in \mathcal{U}$ and $(v_1,v_2)\in \mathcal{V}_1\times\mathcal{V}_2$ is the pair of the followers. Throughout this paper, we use the letter $C$ for a positive constant that may change from one place to another. In equation \eqref{eqq1.1}, $y\vert_\Gamma$ denotes the trace of the function $y$, and $\partial_\nu y = (\nabla y \cdot \nu)\vert_{\Sigma}$ is the normal derivative of $y$, where $\nu$ is the outer unit normal vector at the boundary $\Gamma$. This normal derivative plays the role of the coupling term between the bulk and surface equations. For more details and the physical interpretation of dynamical boundary conditions, also called Wentzell boundary conditions, we refer to \cite{cocang2008, Gal15, Gol06}. System \eqref{eqq1.1} (with $u_1\equiv u_2\equiv u_3\equiv v_1\equiv v_2\equiv 0$) describes various diffusion phenomena, such as thermal processes. These systems, subject to stochastic disturbances, also account for small independent changes during the heat process. Here, stochastic dynamic boundary equations take into account the dynamic interaction between the domain and the boundary. For further details on the physical model described by such systems, see, for instance, \cite{Pbrrune, chuesBj}. Let us now define some useful differential operators on $\Gamma$:
	\begin{itemize}
		\item The tangential gradient of a function $y_\Gamma:\Gamma\rightarrow\mathbb{R}$  is defined by
		$$
		\nabla_{\Gamma} y_\Gamma=\nabla y- \partial_{\nu}y\,\nu,
		$$
		where $y$ is an extension of $y_\Gamma$ up to an open neighborhood of $\Gamma$. On the other hand, one can see that $\nabla_{\Gamma} y_\Gamma(x)$ is the projection of $\nabla y(x)$ onto the tangent space $$T_{x}\Gamma=\{p\in\mathbb{R}^N,\;\; p\,\cdot\,\nu(x)=0\},$$
		at the point $x\in\Gamma$.
		\item The tangential divergence of a function $y_\Gamma:\Gamma\rightarrow\mathbb{R}^N$ such that $y=y_\Gamma$ on $\Gamma$ is defined as
		$$
		\operatorname{div}_{\Gamma} (y_\Gamma)=\operatorname{div}(y)-\nabla y\,\nu\cdot\nu, 
		$$
		where $\nabla y$ is the standard Jacobian matrix of $y$. Then, we define the Laplace-Beltrami operator of a function $y_\Gamma:\Gamma\rightarrow\mathbb{R}$ as follows
		$$
		\Delta_{\Gamma} y_\Gamma= \operatorname{div}_{\Gamma}(\nabla_{\Gamma}y_\Gamma).
		$$
	\end{itemize}
	
	As the standard Sobolev spaces $H^1(G)$ and $H^2(G)$, we also define the following Sobolev spaces $H^1(\Gamma)$ and $H^2(\Gamma)$ on $\Gamma$ by 
	$$H^1(\Gamma)=\Big\{y_\Gamma\in L^2(\Gamma),\;\;\nabla_\Gamma y_\Gamma\in L^2(\Gamma;\mathbb{R}^N)\Big\},
	$$
	and
	$$H^2(\Gamma)=\Big\{y_\Gamma\in L^2(\Gamma),\;\;\nabla_\Gamma y_\Gamma\in H^1(\Gamma;\mathbb{R}^N)\Big\},$$
	endowed with the following norms respectively
	$$ |y_\Gamma|_{H^1(\Gamma)}= \langle y_\Gamma,y_\Gamma \rangle^{\frac{1}{2}}_{H^1(\Gamma)}, \quad\text{with} \quad\,\,\langle y_\Gamma,z_\Gamma \rangle_{H^1(\Gamma)}= \int_{\Gamma}y_\Gamma z_\Gamma \,d\sigma + \int_{\Gamma}\nabla_{\Gamma} y_\Gamma\cdot \nabla_{\Gamma}z_\Gamma \,d\sigma, \quad $$
	and 
	$$ |y_\Gamma|_{H^2(\Gamma)}= \langle y_\Gamma,y_\Gamma \rangle^{\frac{1}{2}}_{H^2(\Gamma)},\quad \text{with}\quad\,\, \langle y_\Gamma,z_\Gamma \rangle_{H^2(\Gamma)}= \int_{\Gamma}y_\Gamma z_\Gamma \,d\sigma +\int_{\Gamma}\Delta_{\Gamma} y_\Gamma\, \Delta_{\Gamma}z_\Gamma \,d\sigma.$$              
	In the subsequent sections, for the Laplace-Belrami operator $\Delta_\Gamma$, we commonly use the following surface divergence formula
	$$
	\int_\Gamma \triangle_\Gamma y_\Gamma\, z_\Gamma\, \,d\sigma = -\int_\Gamma \nabla_\Gamma y_\Gamma\cdot\nabla_\Gamma z_\Gamma\, \,d\sigma\,,\qquad y_\Gamma\in H^2(\Gamma),\;\; z_\Gamma\in H^1(\Gamma).
	$$   
	We also define the following Hilbert spaces
	$$\mathbb{H}^k=\Big\{(y,y_\Gamma)\in H^k(G)\times H^k(\Gamma): \,\,\; y_\Gamma=y|_\Gamma\Big\},\quad k=1,2,$$
	viewed as a subspace of $H^k(G)\times H^k(\Gamma)$ with the natural topology inherited from $H^k(G)\times H^k(\Gamma)$.\\
	
	This paper deals with Stackelberg-Nash strategies for equation \eqref{eqq1.1}. To be more specific, let $y_{i,d}\in\mathcal{V}_{i,d}$ ($i=1,2$) be two fixed target functions. We define the main cost-functional as follows
	\begin{equation*}
		J(u_1,u_2,u_3)=\displaystyle\frac{1}{2}\mathbb{E}\int_{Q} \left(|\chi_{G_0}u_1|^2+|u_2|^2\right) \,dx\, dt+\displaystyle\frac{1}{2}\mathbb{E}\int_{\Sigma} |u_3|^2 \,d\sigma\, dt,
	\end{equation*}
	and the secondary cost-functionals
	\begin{equation}\label{Ji.12}
		J_i(u_1,u_2,u_3,v_1,v_2)=\displaystyle\frac{\alpha_i}{2}\mathbb{E}\int_{ (0,T)\times G_{i,d}} |y-y_{i,d}|^2 \,dx\, dt + \displaystyle\frac{\beta_i}{2}\mathbb{E}\int_{(0,T)\times G_i} |v_i|^2 \,dx\, dt,\quad i=1,2,
	\end{equation}
	where $\alpha_i$, $\beta_i$  are positive constants and $Y(u_1,u_2,u_3,v_1,v_2)=(y(u_1,u_2,u_3,v_1,v_2),y_{\Gamma}(u_1,u_2,u_3,v_1,v_2))$ is the solution of \eqref{eqq1.1}. For a fixed $(u_1,u_2,u_3)\in \mathcal{U}$, the pair $(v^{\star}_1,v^{\star}_2)\in\mathcal{V}_1\times\mathcal{V}_2$ is called a Nash equilibrium for $(J_1,J_2)$ associated to $(u_1,u_2,u_3)$  if 
	\begin{equation*}
		\left\{\begin{array}{ll}
			{J_1(u_1,u_2,u_3,  v^{\star}_1,v^{\star}_2)= \min\limits_{v\in  \mathcal{V}_1}J_1(u_1,u_2,u_3, v,v^{\star}_2),} \\
			{J_2(u_1,u_2,u_3, v^{\star}_1,v^{\star}_2)= \min\limits_{v\in  \mathcal{V}_2}J_2(u_1,u_2,u_3,  v^{\star}_1,v).} 
		\end{array}\right.
	\end{equation*}
	Since the functionals $J_1$ and  $J_2$  are differentiable and  convex, then the pair $(v^{\star}_1,v^{\star}_2)$ is a Nash equilibrium for $(J_1,J_2)$   if and only if
	\begin{equation}\label{NE7}
		\left\{\begin{array}{ll}
			{\displaystyle\frac{\partial J_1}{\partial v_1}(u_1,u_2,u_3,v^{\star}_1,v^{\star}_2)(v)=0, \qquad \forall \, v\in  \mathcal{V}_1,} \\\\
			{\displaystyle\frac{\partial J_2}{\partial v_2}(u_1,u_2,u_3,v^{\star}_1,v^{\star}_2)(v)=0, \qquad \forall \, v\in  \mathcal{V}_2.} 
		\end{array}\right.
	\end{equation}
	
	The goal is to prove that for any initial data  $Y_0=(y_0, y_{\Gamma,0})\in L^2_{\mathcal{F}_0}(\Omega;\mathbb{L}^2)$, there exist controls  $(u_1,u_2,u_3)\in  \mathcal{U}$ minimizing the functional $J$ and an associated Nash equilibrium   $(v^{\star}_1,v^{\star}_2)=(v^{\star}_1(u_1,u_2,u_3),v^{\star}_2(u_1,u_2,u_3))\in\mathcal{V}_1\times\mathcal{V}_2$  such that the associated solution $Y=Y(Y_0,u_1,u_2,u_3, v^{\star}_1,v^{\star}_2)$ of equation \eqref{eqq1.1} satisfies that
	$$Y(T,\cdot)=(0,0)\;\;\textnormal{in}\;\;G\times\Gamma,\quad\mathbb{P}\textnormal{-a.s.}$$ 
	To do this, we shall follow the Stackelberg-Nash strategy: For each choice of the leaders $(u_1,u_2,u_3)$, we look for a Nash equilibrium pair for the cost functionals $J_i$  ($i=1,2$), which means finding the controls $v^{\star}_1(u_1,u_2,u_3)\in \mathcal{V}_1$ and $v^{\star}_2(u_1,u_2,u_3)\in \mathcal{V}_2$, depending on $(u_1,u_2,u_3)$ and satisfying \eqref{NE7}. Once the Nash equilibrium has been identified and fixed for each $(u_1,u_2,u_3)$, we then look for controls $(\widehat{u}_1,\widehat{u}_2,\widehat{u}_3)$ such that
	\begin{equation*}\label{eq1.6}
		J(\widehat{u}_1,\widehat{u}_2,\widehat{u}_3,  v^{\star}_1,v^{\star}_2)= \min_{(u_1,u_2,u_3)\in \mathcal{U}} J(u_1,u_2,u_3, v^{\star}_1(u_1,u_2,u_3), v^{\star}_2(u_1,u_2,u_3)),
	\end{equation*}
	subject to  the null controllability constraint
	\begin{equation*}\label{eq1.7}
		Y(Y_0,\widehat{u}_1,\widehat{u}_2,\widehat{u}_3, v^{\star}_1(\widehat{u}_1,\widehat{u}_2,\widehat{u}_3), v^{\star}_2(\widehat{u}_1,\widehat{u}_2,\widehat{u}_3))(T)=(0,0)\;\;\textnormal{in}\;\;G\times\Gamma,\quad\mathbb{P}\textnormal{-a.s.}
	\end{equation*} 
	As we shall see, the whole problem can be  reduced to one of controllability of a coupled forward-backward stochastic system, which is proved by appropriate Carleman estimates.\\
	
	The main result of this paper is stated as follows: Assume that the control regions satisfy the following assumption:
	\begin{equation}\label{Assump10}
		G_{d}=G_{1,d}=G_{2,d}, \quad \text{and }\quad G_d\cap G_0\neq\emptyset.
	\end{equation}
	\begin{thm}\label{th4.1SN}
		Let us assume that \eqref{Assump10} holds and $\beta_i>0$, $i=1,2$,  are sufficiently large. Then, there exists a positive weight function $\rho =\rho(t)$ blowing up at $t=T$ such that for all the target functions $y_{i,d}\in \mathcal{V}_{i,d}$ satisfying 
		\begin{equation}\label{inqAss11SN}
			\mathbb{E}\int_{(0, T)\times G_{i, d}}\rho^2 |y_{i,d}|^2 \,dx\,dt < \infty,\qquad i=1,2,
		\end{equation}
		and for every initial condition $(y_0,y_{\Gamma, 0} )\in L^2_{\mathcal{F}_0}(\Omega;\mathbb{L}^2)$, there exist controls $(\widehat{u}_1,\widehat{u}_2,\widehat{u}_3)\in \mathcal{U}$  minimizing the functional  $J$ and an associated Nash equilibrium $(v^{\star}_1,v^{\star}_2)\in\mathcal{V}_1\times\mathcal{V}_2$ such that the corresponding solution $(y,y_\Gamma)$ of  \eqref{eqq1.1} satisfies that
		$$(y(T,\cdot),y_\Gamma(T,\cdot))=(0,0)\;\;\textnormal{in}\;\;G\times\Gamma,\quad\mathbb{P}\textnormal{-a.s.}$$
	\end{thm}
	\section{Well-posedness of equations}\label{section3}
	In this section, we provide the well-posedness results of the forward and backward stochastic parabolic equations with dynamic boundary conditions. We refer to \cite{da2014stochastic, krylov, lu2021mathematical, pardoux1990adapted, X. Zhou} for the well-posedness and regularity of solutions of stochastic evolution equations and backward stochastic evolution equations. Let us introduce the following linear operator
	$$\mathcal{A}=\begin{pmatrix} 
		\Delta & 0 \\
		-\partial_{\nu} & \Delta_\Gamma
	\end{pmatrix},\quad\textnormal{with domain}\quad\mathcal{D}(\mathcal{A})=\mathbb{H}^2.$$
	We have the following generation result. For the proof, see \cite{maniar2017null}.
	\begin{prop}\label{prop2.1}
		The operator $\mathcal{A}$ is a densely defined, self-adjoint, dissipative and generates an analytic $C_0$-semigroup $(e^{t\mathcal{A}})_{t\geq0}$ on $\mathbb{L}^2$.
	\end{prop}
	Let us start by discussing the well-posedness of the following  forward stochastic parabolic equation
	\begin{equation}\label{eqqwf1}
		\begin{cases}
			\begin{array}{ll}
				dy - \Delta y \,dt = (a_1y+f_1) \,dt + (a_2y +f_2)\,dW(t)&\textnormal{in}\,\,Q,\\
				dy_\Gamma-\Delta_\Gamma y_\Gamma \,dt+\partial_\nu y \,dt = (b_1y_\Gamma+g_1) \,dt+(b_2y_\Gamma + g_2) \,dW(t) &\textnormal{on}\,\,\Sigma,\\
				y_\Gamma(t,x)=y\vert_\Gamma(t,x) &\textnormal{on}\,\,\Sigma,\\
				(y,y_\Gamma)\vert_{t=0}=(y_0,y_{\Gamma,0}) &\textnormal{in}\,\,G\times\Gamma,
			\end{array}
		\end{cases}
	\end{equation}
	where $(y_0,y_{\Gamma,0})\in L^2_{\mathcal{F}_0}(\Omega;\mathbb{L}^2)$ is the initial state, $(y,y_\Gamma)$ is the state variable, $a_1, a_2\in L_\mathcal{F}^\infty(0,T;L^\infty(G))$, $b_1, b_2\in L_\mathcal{F}^\infty(0,T;L^\infty(\Gamma))$, $f_1,f_2\in L^2_\mathcal{F}(0,T;L^2(G))$ and $g_1,g_2\in L^2_\mathcal{F}(0,T;L^2(\Gamma))$.\\
	Notice that the system \eqref{eqqwf1} can be rewritten in the following abstract stochastic Cauchy problem
	\begin{equation}\label{abs1.1}
		\begin{cases}
			dY=(\mathcal{A}Y+B_1(t)Y + F_1(t)) \,dt + (B_2(t)+F_2(t))Y \,dW(t)\,,\\
			Y(0)=Y_0\in L^2_{\mathcal{F}_0}(\Omega;\mathbb{L}^2),
	\end{cases}\end{equation}
	where 
	$$Y=
	\begin{pmatrix}
		y\\
		y_\Gamma
	\end{pmatrix},\; Y_0=
	\begin{pmatrix}
		y_0\\
		y_{\Gamma,0}
	\end{pmatrix},\; B_1=
	\begin{pmatrix}
		a_1&0\\
		0&b_1
	\end{pmatrix},\; B_2=
	\begin{pmatrix}
		a_2&0\\
		0&b_2
	\end{pmatrix},\; F_1=
	\begin{pmatrix}
		f_1\\
		g_1
	\end{pmatrix},\; F_2=
	\begin{pmatrix}
		f_2\\
		g_2
	\end{pmatrix}.
	$$
	Now, using Proposition \ref{prop2.1} and \cite[Theorem 3.24]{lu2021mathematical}, it is easy to deduce the following well-posedness result of equation \eqref{eqqwf1}.
	\begin{thm}\label{wellposedness1}
		For each  $(y_0,y_{\Gamma,0})\in L^2_{\mathcal{F}_0}(\Omega;\mathbb{L}^2)$, $f_1,f_2\in L^2_{\mathcal{F}}(0,T;L^2(G))$, and $g_1,g_2\in L^2_{\mathcal{F}}(0,T;L^2(\Gamma))$, equation \eqref{abs1.1} $($hence, \eqref{eqqwf1}$)$ admits a unique mild solution
		$$Y=(y,y_\Gamma)\in L^2_\mathcal{F}(\Omega;C([0,T];\mathbb{L}^2))\bigcap L^2_\mathcal{F}(0,T;\mathbb{H}^1),$$
		such that for all $t\in[0,T]$
		$$Y(t)=e^{t\mathcal{A}}Y_0+\int_0^t e^{(t-s)\mathcal{A}}\left[B_1(s)Y(s)+F_1(s)\right]\,ds+\int_0^t e^{(t-s)\mathcal{A}} \left[B_2(s)+F_2(s)\right] \,dW(s),\quad\mathbb{P}\textnormal{-a.s.} $$
		Moreover, there exists a constant $C>0$ depending only on $G$, $T$, $a_1, \,a_2, \,b_1$, and $b_2$ so that
		\begin{align}\label{forwenergyes}
			\begin{aligned}
				&\,\vert(y,y_\Gamma)\vert_{L^2_\mathcal{F}(\Omega;C([0,T];\mathbb{L}^2))} + \vert(y,y_\Gamma)\vert_{L^2_\mathcal{F}(0,T;\mathbb{H}^1)}\\
				&\leq C\,\Big(|(y_0,y_{\Gamma,0})|_{L^2_{\mathcal{F}_0}(\Omega;\mathbb{L}^2)}+|f_1|_{L^2_{\mathcal{F}}(0,T;L^2(G))}+|f_2|_{L^2_{\mathcal{F}}(0,T;L^2(G))}\\
				&\hspace{0.83cm}\,+|g_1|_{L^2_{\mathcal{F}}(0,T;L^2(\Gamma))}+|g_2|_{L^2_{\mathcal{F}}(0,T;L^2(\Gamma))}\Big).
			\end{aligned}
		\end{align}
	\end{thm}
	On the other hand, we consider the following backward stochastic parabolic equation
	\begin{equation}\label{eqqwb1}
		\begin{cases}
			\begin{array}{ll}
				dz + \Delta z \,dt = (a_3z+a_4Z+f_3) \,dt + Z\,dW(t)&\textnormal{in}\,\,Q,\\
				dz_\Gamma+\Delta_\Gamma z_\Gamma \,dt-\partial_\nu z \,dt = (b_3z_\Gamma+b_4\widehat{Z}+g_3) \,dt+\widehat{Z} \,dW(t) &\textnormal{on}\,\,\Sigma,\\
				z_\Gamma(t,x)=z\vert_\Gamma(t,x) &\textnormal{on}\,\,\Sigma,\\
				(z,z_\Gamma)\vert_{t=T}=(z_T,z_{\Gamma,T}) &\textnormal{in}\,\,G\times\Gamma,
			\end{array}
		\end{cases}
	\end{equation}
	where $(z_T,z_{\Gamma,T})\in L^2_{\mathcal{F}_T}(\Omega;\mathbb{L}^2)$ is the terminal state, $(z,z_\Gamma,Z,\widehat{Z})$ is the state variable, $a_3, a_4\in L_\mathcal{F}^\infty(0,T;L^\infty(G))$, $b_3, b_4\in L_\mathcal{F}^\infty(0,T;L^\infty(\Gamma))$, $f_3\in L^2_\mathcal{F}(0,T;L^2(G))$ and $g_3\in L^2_\mathcal{F}(0,T;L^2(\Gamma))$.
	
	In the same manner as \eqref{eqqwf1}, we can rewrite equation \eqref{eqqwb1} as the following abstract backwrd stochastic Cauchy problem
	\begin{equation}\label{absback}
		\begin{cases}
			d\mathcal{Z}=(-\mathcal{A}\mathcal{Z}+B_3(t)\mathcal{Z}+B_4(t)\widehat{\mathcal{Z}}+F_3(t)) \,dt + \widehat{\mathcal{Z}} \,dW(t)\,,\\
			\mathcal{Z}(T)=\mathcal{Z}_T\in L^2_{\mathcal{F}_T}(\Omega;\mathbb{L}^2),
	\end{cases}\end{equation}
	where $$\mathcal{Z}=
	\begin{pmatrix}
		z\\
		z_\Gamma
	\end{pmatrix},\; \mathcal{Z}_T=
	\begin{pmatrix}
		z_T\\
		z_{\Gamma,T}
	\end{pmatrix},\;
	\widehat{\mathcal{Z}}=
	\begin{pmatrix}
		Z\\
		\widehat{Z}
	\end{pmatrix},\; B_3=
	\begin{pmatrix}
		a_3&0\\
		0&b_3
	\end{pmatrix},\; B_4=
	\begin{pmatrix}
		a_4&0\\
		0&b_4
	\end{pmatrix},\; F_3=
	\begin{pmatrix}
		f_3\\
		g_3
	\end{pmatrix}.
	$$
	From Proposition \ref{prop2.1} and \cite[Theorem 4.11]{lu2021mathematical}, we deduce the following well-posedness of \eqref{eqqwb1}.
	\begin{thm}\label{wellposedness2}
		For each  $(z_T, z_{\Gamma,T}) \in L_{\mathcal{F}_T}^2\left(\Omega ; \mathbb{L}^2\right)$, $f_3\in L^2_\mathcal{F}(0,T;L^2(G))$ and $g_3\in L^2_\mathcal{F}(0,T;L^2(\Gamma))$, the equation \eqref{absback} $($hence, \eqref{eqqwb1}$)$ has a unique mild solution
		$$
		(\mathcal{Z},\widehat{\mathcal{Z}})=(z, z_{\Gamma}, Z, \widehat{Z}) \in\left(L_{\mathcal{F}}^2\left(\Omega ; C([0, T]; \mathbb{L}^2)\right) \bigcap L_{\mathcal{F}}^2\left(0, T ; \mathbb{H}^1\right)\right) \times L_{\mathcal{F}}^2\left(0, T ; \mathbb{L}^2\right),
		$$
		so that for all $t\in[0,T]$
		$$\mathcal{Z}(t)=e^{(T-t)\mathcal{A}}\mathcal{Z}_T-\int_t^T e^{(s-t)\mathcal{A}} \Big[B_3(s)\mathcal{Z}(s)+B_4(s)\widehat{\mathcal{Z}}(s)+F_3(s)\Big]\,ds-\int_t^T e^{(s-t)\mathcal{A}} \widehat{\mathcal{Z}}(s) \,dW(s), \;\;\mathbb{P}\textnormal{-a.s.}$$
		Furthermore, there exists a constant $C>0$ depending only on $G$, $T$, $a_3, \,a_4, \,b_3$, and $b_4$  such that
		\begin{align}\label{wbac2.1}
			\begin{aligned}
				&\,\left|\left(z, z_{\Gamma}\right)\right|_{L_{\mathcal{F}}^2\left(\Omega ; C\left([0, T] ; \mathbb{L}^2\right)\right)}+\left|\left(z, z_{\Gamma}\right)\right|_{L_{\mathcal{F}}^2\left(0, T ; \mathbb{H}^1\right)}+|(Z, \widehat{Z})|_{L_{\mathcal{F}}^2\left(0, T ; \mathbb{L}^2\right)}\\
				&\leq C\Big(\left|\left(z_T, z_{\Gamma, T}\right)\right|_{L_{\mathcal{F}_T}^2\left(\Omega ; \mathbb{L}^2\right)}+|f_3|_{L^2_\mathcal{F}(0,T;L^2(G))} +|g_3|_{L^2_\mathcal{F}(0,T;L^2(\Gamma))}\Big).
			\end{aligned}
		\end{align}
	\end{thm}
	\section{Existence, uniqueness, and characterization of Nash-equilibrium}\label{section4}
	Let us first show the following results  concerning the existence and uniqueness of Nash-equilibrium for functionals $J_1$ and $J_2$ given by \eqref{Ji.12}. Here, we denote by $\mathcal{V}=\mathcal{V}_1\times\mathcal{V}_2$.
	\begin{prop}\label{propp4.1}
		There exists a constant $\beta_0 >0$ such that, if $\beta_i\geq \beta_0$ for $i = 1, 2$, then for   each   $(u_1,u_2,u_3)\in\mathcal{U}$, there exists a unique Nash-equilibrium $(v^{\star}_1,v^{\star}_2)=(v^{\star}_1(u_1,u_2,u_3),v^{\star}_2(u_1,u_2,u_3))\in \mathcal{V}$ for $(J_1, J_2)$ associated to $(u_1,u_2,u_3)$. Furthermore, there exists a constant $C>0 $  such that
		\begin{equation}\label{propine4.1}
			|(v^{\star}_1,v^{\star}_2)|_{\mathcal{V}}\leq C\big(1+ |(u_1,u_2,u_3)|_{\mathcal{U}}\big).
		\end{equation}
	\end{prop}
	\begin{proof}
		Define the bounded operators $L_i\in \mathcal{L}(\mathcal{V}_i; L_\mathcal{F}^2(0,T; \mathbb{L}^2))$ and  $\ell_i\in \mathcal{L}(\mathcal{V}_i; L_\mathcal{F}^2(0,T; L^2(G)))$  by
		\begin{align*}
			L_i(v_i)=Y^i= (y^i, y_{\Gamma}^i),\,\quad\text{and}\,\quad \ell_i(v_i)=y^i,\qquad i=1,2,
		\end{align*}
		where $Y^i=(y^i, y_{\Gamma}^i)$ is the solution of the following equation
		\begin{equation}\label{1.1}
			\begin{cases}
				\begin{array}{ll}
					dy^{i} - \Delta y^{i} \,dt = [a_1 y^{i}+\chi_{G_i}(x)v_i] \,dt + a_2 y^{i} \,dW(t)&\textnormal{in}\,\,Q,\\
					dy^{i}_\Gamma-\Delta_\Gamma y^{i}_\Gamma \,dt+\partial_\nu y^{i} \,dt = b_1 y^{i}_\Gamma \,dt+ b_2 y^{i}_\Gamma   \,dW(t) &\textnormal{on}\,\,\Sigma,\\
					y^{i}_\Gamma(t,x)=y^{i}\vert_\Gamma(t,x) &\textnormal{on}\,\,\Sigma,\\
					(y^{i},y^{i}_\Gamma)\vert_{t=0}=(0,0) &\textnormal{in}\,\,G\times\Gamma,\\\
					i=1,2.
				\end{array}
			\end{cases}
		\end{equation}
		By linearity, it is easy to see that the solution $Y$ of \eqref{eqq1.1} can be expressed as follows
		$$Y(u_1,u_2,u_3,v_1,v_2)=L_1(v_1)+L_2(v_2)+L(u_1,u_2,u_3),$$
		where $L(u_1,u_2,u_3)=(q, q_{\Gamma})$ solves the system
		\begin{equation*}\label{1.132}
			\begin{cases}
				\begin{array}{ll}
					dq - \Delta q \,dt = [a_1 q+\chi_{G_0}(x)u_1] \,dt + (a_2 q+u_2)\,dW(t)&\textnormal{in}\,\,Q,\\
					dq_\Gamma-\Delta_\Gamma q_\Gamma \,dt+\partial_\nu q \,dt = b_1 q_\Gamma \,dt+ (b_2 q_\Gamma +u_3) \,dW(t) &\textnormal{on}\,\,\Sigma,\\
					q_\Gamma(t,x)=q\vert_\Gamma(t,x) &\textnormal{on}\,\,\Sigma,\\
					(q,q_\Gamma)\vert_{t=0}=(y_0,y_{\Gamma,0})&\textnormal{in}\,\,G\times\Gamma.
				\end{array}
			\end{cases}
		\end{equation*}
		For fixed $(u_1,u_2,u_3)\in\mathcal{U}$, we have that  for all $v_i\in \mathcal{V}_{i}$ ($i=1,2$)
		\begin{align*}
			\frac{\partial J_i}{\partial v_i}(u_1,u_2,u_3; v^{\star}_1,v^{\star}_2)(v_i) =&\,\,\alpha_i \big\langle \ell_1(v^{\star}_1)+\ell_2(v^{\star}_2) +q(u_1,u_2,u_3)-y_{i,d},	\ell_i(v_i)\big\rangle_{\mathcal{V}_{i,d}}+ \beta_i \big\langle v^{\star}_i,v_i\big\rangle_{\mathcal{V}_{i}}.
		\end{align*}
		Thus, $(v^{\star}_1,v^{\star}_2)$ is a Nash-equilibrium if and only if	
		\begin{align}\label{4.3nashcara}
			\alpha_i \big\langle\ell_i^*\big[\ell_1(v^{\star}_1)+\ell_2(v^{\star}_2) -\widetilde{y}_{i,d}\big],	v_i\big\rangle_{\mathcal{V}_{i,d}}+ \beta_i \big\langle v^{\star}_i,v_i\big\rangle_{\mathcal{V}_{i}}=0,\quad\forall v_i\in \mathcal{V}_{i}, \qquad i=1,2,
		\end{align}
		where $\widetilde{y}_{i,d}=y_{i,d}-q$. Therefore, $(v^{\star}_1,v^{\star}_2)$ is a Nash-equilibrium if and only if
		$$\alpha_i \ell_i^*\left[\ell_1(v^{\star}_1)\chi_{G_{i,d}}+\ell_2(v^{\star}_2)\chi_{G_{i,d}}\right]+ \beta_i v^{\star}_i=\alpha_i\ell_i^*(\widetilde{y}_{i,d}\chi_{G_{i,d}}),\qquad i=1,2.$$
		Now, let us define the following operator $\textbf{M}\in\mathcal{L}(\mathcal{V};\mathcal{V})$ by
		$$\textbf{M}(v_1,v_2)=\Big(\alpha_1 \ell_1^*\left[\ell_1(v_1)\chi_{G_{1,d}}+\ell_2(v_2)\chi_{G_{1,d}}\right]+ \beta_1 v_1\,,\,\alpha_2 \ell_2^*[\ell_1(v_1)\chi_{G_{2,d}}+\ell_2(v_2)\chi_{G_{2,d}}]+ \beta_2 v_2\Big),$$
		for all $(v_1,v_2)\in\mathcal{V}$, and introduce the following bi-linear functional $\textbf{a}:\mathcal{V}\times\mathcal{V}\rightarrow\mathbb{R}$ as
		$$\textbf{a}((v_1,v_2),(\widetilde{v}_1,\widetilde{v}_2))=\langle \textbf{M}(v_1,v_2),(\widetilde{v}_1,\widetilde{v}_2)\rangle_\mathcal{V},\qquad\forall \big((v_1,v_2),(\widetilde{v}_1,\widetilde{v}_2)\big)\in\mathcal{V}\times\mathcal{V}.$$
		We also define the following linear continuous functional $\Psi:\mathcal{V}\rightarrow\mathbb{R}$ by
		$$\Psi(v_1,v_2)=\big\langle(v_1,v_2),(\alpha_1\ell_1^*(\widetilde{y}_{1,d}\chi_{1,d}),\alpha_2\ell_2^*(\widetilde{y}_{2,d}\chi_{2,d}))\big\rangle_\mathcal{V},\qquad\forall(v_1,v_2)\in\mathcal{V}.$$
		Following some ideas from the proof of \cite[Proposition 4]{BoMaOuNash}, it is not difficult to see that $\textbf{a}$ is continuous and coercive. Now, by Lax-Milgram theorem, it is easy to see that there exists a unique $(v^*_1,v^*_2)\in\mathcal{V}$ such that
		\begin{align}\label{4.4105}
			\textbf{a}((v^*_1,v^*_2),(v_1,v_2))=\Psi(v_1,v_2),\qquad\forall(v_1,v_2)\in\mathcal{V}.
		\end{align}
		This easily implies the existence and uniqueness of the Nash equilibrium $(v^*_1,v^*_2)$ for $(J_1,J_2)$ associated to $(u_1,u_2,u_3)$, and from \eqref{4.4105}, we deduce that 
		$$|(v^*_1,v^*_2)|_\mathcal{V}\leq C|(\alpha_1\ell_1^*(\widetilde{y}_{1,d}\chi_{1,d}),\alpha_2\ell_2^*(\widetilde{y}_{2,d}\chi_{2,d}))|_\mathcal{V}.$$
		Thus, this provides the desired estimates \eqref{propine4.1}.
	\end{proof}
	Now, we  characterize the Nash equilibrium by the following adjoint system  
	\begin{equation}\label{backadj}
		\begin{cases} 
			\begin{array}{ll}
				dz^{i}+\Delta z^{i} \,dt=\big[-a_1 z^{i}-a_2 Z^{i}-\alpha_i(y-y_{i,d})\chi_{G_{i,d}}(x)\big] \,dt+Z^{i} \,dW(t) &\textnormal{in}\,\,Q, \\ dz^{i}_{\Gamma}+\Delta_\Gamma z^{i}_\Gamma \,dt-\partial_\nu z^i \,dt=\big[-b_1 z^{i}_{\Gamma}-b_2 \widehat{Z}^{i}\big] \,dt+\widehat{Z}^{i} \,dW(t) &\textnormal{on}\,\,\Sigma, \\ z^i_{\Gamma}(t, x)=z^{i}|_\Gamma(t, x) &\textnormal{on}\,\,\Sigma, \\ (z^i, z^{i}_\Gamma)|_{t=T}=(0,0) &\textnormal{in}\,\,G \times \Gamma,\\
				i=1,2.
			\end{array}
		\end{cases}
	\end{equation}
	Multiplying solutions of \eqref{1.1} and \eqref{backadj} and integrating by parts, we find that
	\begin{equation}\label{4.662}
		\alpha_i\big\langle y(u_1,u_2,u_3,v_1,v_2)-y_{i,d}, \ell_i(v_i)\big\rangle_{\mathcal{V}_{i,d}}= \langle z^i, v_i\rangle_{\mathcal{V}_i},\qquad i=1,2.
	\end{equation}
	Combining \eqref{4.3nashcara} and \eqref{4.662}, we deduce that, $(v^{\star}_1,v^{\star}_2)$ is a Nash-equilibrium if and only if 
	$$\langle z^i, v_i\rangle_{\mathcal{V}_i}+\beta_i\langle v^{\star}_i, v_i\rangle_{\mathcal{V}_i}=0, \quad  \forall v_i\in \mathcal{V}_i,\qquad i=1,2.$$
	Therefore
	\begin{equation*}\label{chara.1}
		v^{\star}_i=-\frac{1}{\beta_i} z^i|_{(0,T)\times G_i}, \qquad i=1,2.
	\end{equation*}
	
	Using the previous results, we now obtain the following forward-backward system, called the optimality system.
	\begin{equation}\label{eqq4.7}
		\begin{cases}
			\begin{array}{ll}
				dy - \Delta y \,dt = \Big[a_1y+\chi_{G_0}(x)u_1-\frac{1}{\beta_1}\chi_{G_1}(x)z^1-\frac{1}{\beta_2}\chi_{G_2}(x)z^2\Big] \,dt + (a_2y +u_2)\,dW(t)&\textnormal{in}\,\,Q,\\
				dy_\Gamma-\Delta_\Gamma y_\Gamma \,dt+\partial_\nu y \,dt = b_1y_\Gamma \,dt+(b_2y_\Gamma  +\,u_3) \,dW(t)&\textnormal{on}\,\,\Sigma,\\
				dz^{i}+\Delta z^{i} \,dt=\big[-a_1 z^{i}-a_2 Z^{i}-\alpha_i(y-y_{i,d})\chi_{G_{i,d}}(x)\big] \,dt+Z^{i} \,dW(t)&\textnormal{in}\,\,Q, \\ 
				dz^{i}_{\Gamma}+\Delta_\Gamma z^{i}_\Gamma \,dt-\partial_\nu z^i \,dt=\big[-b_1 z^{i}_{\Gamma}-b_2 \widehat{Z}^{i}\big] \,dt+\widehat{Z}^{i} \,dW(t)&\textnormal{on}\,\,\Sigma, \\ 
				y_\Gamma(t,x)=y\vert_\Gamma(t,x) &\textnormal{on}\,\,\Sigma,\\
				z^i_{\Gamma}(t, x)=z^{i}|_\Gamma(t, x)&\textnormal{on}\,\,\Sigma, \\
				(y, y_\Gamma)|_{t=0}=(y_0, y_{\Gamma, 0}) &\textnormal{in}\,\, G\times \Gamma, \\
				(z^i, z^{i}_\Gamma)|_{t=T}=(0, 0)&\textnormal{in}\,\,G \times \Gamma,\\
				i=1,2.
			\end{array}
		\end{cases}
	\end{equation}
	In the sequel, we will show the null controllability for the system \eqref{eqq4.7} by establishing a suitable observability inequality for the following adjoint system
	\begin{equation}\label{ADJSO1}
		\begin{cases}
			d\phi+\Delta \phi \,dt=\big[-a_1 \phi-a_2\Phi+\alpha_1 \psi^{1}\chi_{G_{1,d}}(x)+\alpha_2\psi^{2}\chi_{G_{2,d}}(x) \big] \,dt+\Phi \,dW(t) &\textnormal{in}\,\,Q, \\ 
			d\phi_{\Gamma}+\Delta_\Gamma \phi_\Gamma \,dt-\partial_\nu \phi \,dt=\big[-b_1\phi_{\Gamma}-b_2 \widehat{\Phi}\big] \,dt+ \widehat{\Phi} \,dW(t) &\textnormal{on}\,\,\Sigma, \\ 
			d\psi^{i} - \Delta \psi^{i}\,dt = [a_1\psi^{i}+\frac{1}{\beta_{i}}\chi_{G_i}(x)\phi] \,dt + a_2 \psi^{i} \,dW(t)&\textnormal{in}\,\,Q,\\
			d\psi^i_\Gamma-\Delta_\Gamma \psi^i_\Gamma \,dt+\partial_\nu \psi^i \,dt = b_1\psi^i_\Gamma \,dt+ b_2\psi^i_\Gamma   \,dW(t) &\textnormal{on}\,\,\Sigma,\\
			\phi_{\Gamma}(t, x)=\phi|_\Gamma(t, x) &\textnormal{on}\,\,\Sigma, \\
			\psi^i_\Gamma(t,x)=\psi^i\vert_\Gamma(t,x) &\textnormal{on}\,\,\Sigma,\\
			(\phi,\phi_\Gamma)|_{t=T}=(\phi_T,\phi_{\Gamma, T}) &\textnormal{in}\,\,G\times\Gamma,\\
			(\psi^i, \psi^i_\Gamma)|_{t=0}=(0, 0) &\textnormal{in}\,\,G\times\Gamma,\\
			i=1,2.
		\end{cases}
	\end{equation}
	\section{Carleman estimates}\label{section5}
	In this section, we prove the needed Carleman estimate for the coupled system \eqref{ADJSO1}. Let us first introduce the following technical result. For the proof, see \cite{BFurIman}.
	\begin{lm}\label{lmm5.1}
		For any nonempty open subset $G'\Subset G$ $($i.e., $\overline{G'}\subset G$$)$, there exists a function $\eta\in C^4(\overline{G})$ such that
		$$
		\eta>0\;\,\, \textnormal{in} \,\,G\,;\qquad \eta=0\;\,\,\, \textnormal{on} \,\,\Gamma;\qquad\vert\nabla\eta\vert>0\; \,\,\,\,\textnormal{in}\,\,\overline{G\setminus G'}.
		$$
	\end{lm}
	In what follows, we particularly consider $G'\Subset G_0$. Since $\vert\nabla\eta\vert^2=\vert\nabla_\Gamma\eta\vert^2+\vert\partial_\nu\eta\vert^2$ on $\Gamma$,  the function $\eta$ in Lemma \ref{lmm5.1} satisfies
	\begin{equation*}
		\nabla_\Gamma\eta=0\,,\qquad \vert\nabla\eta\vert=\vert\partial_\nu\eta\vert\,,\qquad\partial_\nu\eta\leq -c<0\,\,\,\;\;\textnormal{on}\,\,\;\Gamma, \,\,\,\textnormal{for some constant}\,\,c>0.
	\end{equation*}
	For large parameters $\lambda>1$ and $\mu>1$, we choose the following weight functions
	\begin{align*}
		&\,\alpha=\alpha(t,x) = \frac{e^{\mu\eta(x)}-e^{2\mu\vert\eta\vert_\infty}}{t(T-t)},\qquad \varphi=\varphi(t,x) = \frac{e^{\mu\eta(x)}}{t(T-t)},\qquad\theta=e^{\lambda\alpha}.
	\end{align*}
	Moreover, the weight functions $\alpha$ and $\varphi$ are constants on the boundary $\Gamma$, then one has
	\begin{equation*}
		\nabla_\Gamma\alpha=0\qquad\textnormal{and}\qquad\nabla_\Gamma\varphi=0\,\quad \textnormal{on}\;\,\Gamma.
	\end{equation*}
	On the other hand, it is easy to check that there exists a constant $C>0$ depending only on $G$, $G_0$, and  $T$ such that for all $(t,x)\in(0,T)\times\overline{G}$
	\begin{align*}\begin{aligned}
			&\,\varphi(t,x)\geq C,\qquad\vert\varphi_t(t,x)\vert\leq C\varphi^2(t,x),\qquad\vert\varphi_{tt}(t,x)\vert\leq C\varphi^3(t,x),\\
			&\vert\alpha_t(t,x)\vert\leq Ce^{2\mu\vert\eta\vert_\infty}\varphi^2(t,x),\qquad\vert\alpha_{tt}(t,x)\vert\leq Ce^{2\mu\vert\eta\vert_\infty}\varphi^3(t,x).
	\end{aligned}\end{align*}
	To simplify the formulas, we adopt the following notation
	\begin{align*}
		\mathcal{I}(\lambda,\mu; \xi, \xi_{\Gamma})&:= \lambda^3\mathbb{E}\int_Q\theta^2\varphi^3 \xi^2\,dx\,dt+\lambda^3\mathbb{E}\int_\Sigma\theta^2\varphi^3 \xi_\Gamma^2\,d\sigma\,dt\\ &\quad+\lambda\mathbb{E}\int_Q \theta^2\varphi|\nabla\xi|^2\,dx\,dt+\lambda\mathbb{E}\int_\Sigma \theta^2\varphi|\nabla_{\Gamma}\xi_{\Gamma}|^2 \,d\sigma\,dt,
	\end{align*}
	where $(\xi,\xi_{\Gamma})$ (resp., $(\xi,\xi_{\Gamma},\Xi,\widehat{\Xi})$) is the solution of a suitable  forward (resp., backward) stochastic parabolic equation with dynamic boundary conditions. Let us now introduce the following forward stochastic parabolic equation
	\begin{equation}\label{eqqgfr}
		\begin{cases}
			\begin{array}{ll}
				dz - \Delta z \,dt = f \,dt + g\,dW(t)&\textnormal{in}\,\,Q,\\
				dz_\Gamma-\Delta_\Gamma z_\Gamma \,dt+\partial_\nu z \,dt = f_{\Gamma} \,dt+ g_{\Gamma}  \,dW(t) &\textnormal{on}\,\,\Sigma,\\
				z_\Gamma(t,x)=z\vert_\Gamma(t,x) &\textnormal{on}\,\,\Sigma,\\
				(z,z_\Gamma)\vert_{t=0}=(z_0,z_{\Gamma,0}) &\textnormal{in}\,\,G\times\Gamma,
			\end{array}
		\end{cases}
	\end{equation}
	where $(z_0,z_{\Gamma,0})\in L^2_{\mathcal{F}_0}(\Omega;\mathbb{L}^2)$ is the initial state, $f, g\in L^2_\mathcal{F}(0,T;L^2(G))$ and $f_\Gamma,g_\Gamma\in L^2_\mathcal{F}(0,T;L^2(\Gamma))$ are some source terms. Now, we first recall the following Carleman estimate for equation \eqref{eqqgfr}. For the proof, we refer to \cite[Theorem 3.1]{elgrouDBC}.
	\begin{lm}
		There exist a large $\mu_1>1$ such that for all $\mu\geq\mu_1$, one can find constants $C>0$ and $\lambda_1>1$ depending only on $G$, $G_0$, $\mu$ and $T$ such that for all $\lambda\geq\lambda_1$, $f, g\in L^2_\mathcal{F}(0,T;L^2(G))$, $f_\Gamma, g_\Gamma\in L^2_\mathcal{F}(0,T;L^2(\Gamma))$, and $(z_0,z_{\Gamma,0})\in L^2_{\mathcal{F}_0}(\Omega;\mathbb{L}^2)$, the mild solution $(z,z_\Gamma)$ of \eqref{eqqgfr} satisfies 
		\begin{align}\label{carfor5.6}
			\begin{aligned}
				&\;\mathcal{I}(\lambda,\mu;z,z_\Gamma)\leq C \Bigg[ \lambda^3\mathbb{E}\int_{Q_0} \theta^2\varphi^3 z^2 \,dx\,dt+ \mathbb{E}\int_Q \theta^2f^2 \,dx\,dt+\mathbb{E}\int_\Sigma \theta^2 f_\Gamma^2 \,d\sigma dt\\
				&\hspace{2.8cm}+\lambda^2\mathbb{E}\int_Q \theta^2\varphi^2g^2 \,dx\,dt+\lambda^2\mathbb{E}\int_\Sigma \theta^2\varphi^2 g_\Gamma^2 \,d\sigma dt\Bigg]. 
		\end{aligned}\end{align}
	\end{lm}
	On the other hand, we  consider the following backward stochastic parabolic equation
	\begin{equation}\label{eqqgbc}
		\begin{cases}
			dz+\Delta z\,dt=F \,dt+ Z \,dW(t) & \text { in }Q, \\ 
			d z_{\Gamma}+\Delta_\Gamma z_\Gamma \,dt-\partial_\nu z \,dt=F_{\Gamma} \,dt+ \widehat{Z}\,dW(t) & \text { on }\Sigma, \\ 
			z_{\Gamma}(t, x)=z|_\Gamma(t, x) & \text { on }\Sigma, \\
			(z, z_\Gamma)|_{t=T}=(z_T, z_{\Gamma, T}) & \text { in } G\times\Gamma,
		\end{cases}
	\end{equation}
	where $(z_T,z_{\Gamma,T})\in L^2_{\mathcal{F}_T}(\Omega;\mathbb{L}^2)$ is the terminal state, $F\in L^2_\mathcal{F}(0,T;L^2(G))$ and $F_\Gamma\in L^2_\mathcal{F}(0,T;L^2(\Gamma))$. From \cite[Theorem 4.1]{elgrouDBC}, we also recall the following Carleman estimate for equation \eqref{eqqgbc}.
	\begin{lm}
		There exist a large $\mu_2>1$ such that for $\mu\geq\mu_2$, one can find constants $C>0$ and $\lambda_2>1$ depending only on $G$, $G_0$, $\mu$ and $T$ such that for all $\lambda\geq\lambda_2$, $F\in L^2_\mathcal{F}(0,T;L^2(G))$, $F_\Gamma\in L^2_\mathcal{F}(0,T;L^2(\Gamma))$, and $(z_T,z_{\Gamma,T})\in L^2_{\mathcal{F}_T}(\Omega;\mathbb{L}^2)$, the mild solution $(z,z_\Gamma, Z, \widehat{Z})$ of \eqref{eqqgbc} satisfies
		\begin{align}\label{carback5.8}
			\begin{aligned}
				&\;\mathcal{I}(\lambda,\mu;z,z_\Gamma)\leq C \Bigg[ \lambda^3\mathbb{E}\int_{Q_0} \theta^2\varphi^3 z^2 \,dx\,dt+ \mathbb{E}\int_Q \theta^2 F^2 \,dx\,dt+\mathbb{E}\int_\Sigma \theta^2 F_\Gamma^2\, d\sigma dt\\
				&\hspace{2,8cm}+\lambda^2\mathbb{E}\int_Q \theta^2\varphi^2 Z^2 \,dx\,dt+\lambda^2\mathbb{E}\int_\Sigma \theta^2\varphi \widehat{Z}^2 \,d\sigma dt\Bigg]. 
			\end{aligned}
		\end{align}
	\end{lm}
	By combining Carleman estimates \eqref{carfor5.6} and \eqref{carback5.8}, we  will prove the following Carleman estimate for solutions of  the coupled forward-backward system \eqref{ADJSO1}.
	\begin{lm}\label{thmm5.1} 
		Let $((\phi,\phi_{\Gamma},\Phi, \widehat{\Phi}), \Psi^1, \Psi^2)$ be the solution of \eqref{ADJSO1} and assume that \eqref{Assump10} holds. Then, there exist a large $\mu_3>1$ such that for $\mu=\mu_3$, one can find constants $C>0$ and $\lambda_3>1$ depending only on $G$, $G_0$, $\mu_3$ and $T$ such that 
		\begin{align}\label{Carlem5.9}
			\begin{aligned}
				\mathcal{I}(\lambda,\mu; \phi, \phi_{\Gamma})+ \mathcal{I}(\lambda,\mu; h, h_{\Gamma}) \leq &\,\,C \Bigg[\lambda^7  \mathbb{E}\int_{Q_0} \theta^2 \varphi^7 \phi^2 \, \,dx\,dt+\lambda^2 \mathbb{E}\int_{Q}\theta^2\varphi^2\Phi^2\,dx\,dt\\
				&\quad\;\;+\lambda^2 \mathbb{E}\int_{\Sigma}\theta^2\varphi\widehat{\Phi}^2 \,d\sigma\,dt\Bigg],
			\end{aligned}
		\end{align}
		for all $\lambda\geq\lambda_3$, where $(h, h_{\Gamma})= \alpha_1\Psi^1+\alpha_2\Psi^2=\alpha_1(\psi^1, \psi^1_{\Gamma})+\alpha_2(\psi^2, \psi^2_{\Gamma})$.
	\end{lm}
	\begin{proof}
		From Lemma \ref{lmm5.1}, we choose the set $G'$ so that 
		\begin{equation*}\label{Eq31D}
			G'\subset G'_1\subset\subset G_0\cap G_d,
		\end{equation*}
		where $G'_1$ is another arbitrary open subset and we consider $\zeta\in C^{\infty}_0(G)$ such that 
		\begin{equation}\label{assmzeta}
			0\leq\zeta\leq 1,\,\quad \zeta =1 \,\, \text{in}\,\, G'\quad \text{and}\,\quad \text{supp}(\zeta)\subset G'_1,\quad |\Delta\zeta|\leq C\zeta,\;\;|\nabla\zeta|\leq C\zeta.
		\end{equation}
		
		Using the system satisfied by $(h,h_\Gamma)$ and apply the Carleman estimate \eqref{carfor5.6}, we conclude that there exist a large $\mu_1>1$ such that for $\mu\geq\mu_1$, one can find constants $C>0$ and a large enough $\lambda_1>1$ so that for all $\lambda\geq\lambda_1$, we obtain
		\begin{align*}
			\mathcal{I}(\lambda,\mu; h, h_{\Gamma})  &\leq C\Bigg[\lambda^3 \mathbb{E}\int_0^T\int_{G'} \theta^2\varphi^3 h^2 \, dx\,dt+\mathbb{E}\int_{Q}\theta^2\Big|\frac{\alpha_1}{\beta_1}\chi_{G_1}(x)\phi+\frac{\alpha_2}{\beta_2}\chi_{G_2}(x)\phi\Big|^2dx\,dt\Bigg].
		\end{align*}
		Then, it follows that
		\begin{align}\label{Car4.13}
			\mathcal{I}(\lambda,\mu; h, h_{\Gamma})\leq C\Bigg[\lambda^3\mathbb{E}\int_0^T\int_{G'} \theta^2\varphi^3 h^2 \, dx\,dt
			+ \mathbb{E}\int_{Q} \theta^2\phi^2\, dx\,dt\Bigg].
		\end{align}
		On the other hand, using the Carleman estimate \eqref{carback5.8}, we deduce that there exists $\mu_2>1$ such that for $\mu\geq\mu_2$, one can find constants $C>0$ and a large enough $\lambda_2>1$ so that for $\lambda\geq \lambda_2$ 
		\begin{align}\label{car4.14}
			\begin{aligned}
				\mathcal{I}(\lambda,\mu; \phi, \phi_{\Gamma})  &\leq C\Bigg[\lambda^3  \mathbb{E}\int_0^T\int_{G'} \theta^2\varphi^3 \phi^2 \, dx\,dt+\mathbb{E}\int_{Q} \theta^2 h^2 \, dx\,dt+\lambda^2\mathbb{E}\int_{Q} \theta^2\varphi^2 \Phi^2\,dx\,dt\\
				&\quad\quad\;\;+\lambda^2\mathbb{E}\int_{\Sigma} \theta^2\varphi\widehat{\Phi}^2\,d\sigma\,dt\Bigg].
			\end{aligned}
		\end{align}
		Combining \eqref{Car4.13} and \eqref{car4.14} and choosing  $\mu=\mu_3=\max(\mu_1,\mu_2)$ and a large enough $\lambda\geq1$, we get
		\begin{align}\label{firsine1}
			\begin{aligned}
				\mathcal{I}(\lambda,\mu;\phi, \phi_{\Gamma}) +\mathcal{I}(\lambda,\mu; h, h_{\Gamma})  &\leq C\Bigg[ \lambda^3 \mathbb{E}\int_0^T\int_{G'} \theta^2\varphi^3 \phi^2 \,dx\,dt +\lambda^3 \mathbb{E}\int_0^T\int_{G'} \theta^2\varphi^3 h^2 \, dx\,dt\\
				&\quad\;\quad+\lambda^2\mathbb{E}\int_{Q}\theta^2\varphi^2 \Phi^2\,dx\,dt+\lambda^2\mathbb{E}\int_{\Sigma}\theta^2\varphi\widehat{\Phi}^2\,d\sigma\,dt\Bigg].
			\end{aligned}
		\end{align}
		In the sequel, we shall give an estimate for ``$\displaystyle\lambda^3 \mathbb{E}\int_0^T\int_{G'} \theta^2\varphi^3 h^2 \, dx\,dt$''. To this end, using Itô's formula, we compute $d(\lambda^3\zeta \theta^2\varphi^3 h\phi )$. Integrating the obtained identity over $Q$ and taking the mathematical expectation on both sides, we obtain that
		\begin{align*}
			\begin{aligned}
				0&=\mathbb{E}\int_Q d(\lambda^3\zeta \theta^2\varphi^3 h\phi) \,dx\,dt\\
				&=\mathbb{E}\int_Q \lambda^3\zeta\theta^2\varphi^3\Bigg\{\phi\Big[\Delta h+a_1h+\Big(\frac{\alpha_1}{\beta_1}\chi_{G_1}+\frac{\alpha_2}{\beta_2}\chi_{G_2}\Big)\phi\Big]+a_2h\Phi\\
				&\quad\quad\;+h\big[-\Delta\phi-a_1\phi-a_2\Phi+\alpha_1\psi^1\chi_{G_d}+\alpha_2\psi^2\chi_{G_d}\big]\Bigg\}\, dx\,dt\\
				&\quad+\mathbb{E}\int_Q (\lambda^3\zeta\theta^2\varphi^3)_th\phi \,dx\,dt\\
				&=\mathbb{E}\int_Q \lambda^3\zeta\theta^2\varphi^3\Bigg\{\phi\Big[\Delta h+\Big(\frac{\alpha_1}{\beta_1}\chi_{G_1}+\frac{\alpha_2}{\beta_2}\chi_{G_2}\Big)\phi\Big]+h\big[-\Delta\phi+h\chi_{G_d}\big]\Bigg\}\, dx\,dt\\
				&\quad+\mathbb{E}\int_Q (\lambda^3\zeta\theta^2\varphi^3)_th\phi \,dx\,dt.
			\end{aligned}
		\end{align*}
		Then, it follows that 
		\begin{align}\label{integ11}
			\begin{aligned}
				\lambda^3\mathbb{E}\int_0^T\int_{G'_1} \zeta \theta^2\varphi^3 h^2 \,dx\,dt &=\lambda^3\mathbb{E}\int_Q \zeta \theta^2\varphi^3\Big(\frac{\alpha_1}{\beta_1}\chi_{G_1}+\frac{\alpha_2}{\beta_2}\chi_{G_2}\Big)\phi^2 \,dx\,dt-\lambda^3\mathbb{E}\int_Q \zeta(\theta^2\varphi^3)_t h\phi \,dx\,dt\\
				&\quad\;-\lambda^3\mathbb{E}\int_Q \zeta \theta^2\varphi^3\phi \Delta h \,dx\,dt+\lambda^3\mathbb{E}\int_Q \zeta \theta^2\varphi^3h\Delta\phi \,dx\,dt\\
				&:=I_1+I_2+I_3+I_4.
			\end{aligned}
		\end{align}
		Now, let us estimate $I_i$ ($i=1,2,3,4$).
		\begin{itemize}
			\item For $I_1$: it easy to see that
			\begin{align}\label{Int2}
				I_1\leq C\lambda^3\mathbb{E}\int_{Q_0}\theta^2\varphi^3\phi^2 \,dx\,dt.
			\end{align}
			\item For $I_2$: note that for a large enough $\lambda$, we have
			$$|(\theta^2\varphi^3)_t|\leq C\lambda\theta^2\varphi^5.$$
			Then, by Young's inequality, it follows that for any $\varepsilon>0$
			\begin{align}\label{Int4}
				\begin{aligned}
					I_2&\leq C\lambda^4\mathbb{E}\int_0^T\int_{G_1'} \zeta\theta^2\varphi^5|h||\phi| \,dx\,dt\\
					&\leq \varepsilon\lambda^3\mathbb{E}\int_0^T\int_{G_1'}\zeta \theta^2\varphi^3h^2 \,dx\,dt+\frac{C}{\varepsilon}\lambda^5\mathbb{E}\int_{Q_0} \theta^2\varphi^7\phi^2 \,dx\,dt.
				\end{aligned}
			\end{align}
			\item For $I_3+I_4$: by integration by parts and \eqref{assmzeta}, we find that
			\begin{align}\label{ine5.144}
				\begin{aligned}
					I_3+I_4 &= \lambda^3\mathbb{E}\int_Q \Delta(\zeta\theta^2\varphi^3)h\phi\,dx\,dt + 2 \lambda^3\mathbb{E}\int_Q \phi\nabla(\zeta\theta^2\varphi^3)\cdot\nabla h \,dx\,dt\\
					&\leq C\lambda^3\mathbb{E}\int_0^T\int_{G_1'} \zeta\theta^2\varphi^3|h||\phi|\,dx\,dt+C\lambda^3\mathbb{E}\int_0^T\int_{G_1'} |\Delta(\theta^2\varphi^3)|\zeta|h||\phi|\,dx\,dt\\
					&\quad\,+C\lambda^3\mathbb{E}\int_0^T\int_{G_1'} |\nabla(\theta^2\varphi^3)|\zeta|h||\phi|\,dx\,dt+C\lambda^3\mathbb{E}\int_0^T\int_{G_1'} \theta^2\varphi^3|\phi||\nabla h|\,dx\,dt\\
					&\quad\,+C\lambda^3\mathbb{E}\int_0^T\int_{G_1'} |\nabla(\theta^2\varphi^3)||\phi||\nabla h|\,dx\,dt.
				\end{aligned}
			\end{align}
			On the other hand, it is easy to see that for a large enough $\lambda$, we have
			\begin{align}\label{5.15in}
				|\nabla(\theta^2\varphi^3)|\leq C\lambda\theta^2\varphi^4\,,\quad\;|\Delta(\theta^2\varphi^3)|\leq C\lambda^2\theta^2\varphi^5.
			\end{align}
			Using \eqref{5.15in} in \eqref{ine5.144}, then we conclude that
			\begin{align}\label{ine5.16}
				\begin{aligned}
					I_3+I_4&\leq C\lambda^3\mathbb{E}\int_0^T\int_{G_1'} \zeta\theta^2\varphi^3|h||\phi|\,dx\,dt+C\lambda^5\mathbb{E}\int_0^T\int_{G_1'} \zeta\theta^2\varphi^5|h||\phi|\,dx\,dt\\
					&\quad+C\lambda^4\mathbb{E}\int_0^T\int_{G_1'} \zeta\theta^2\varphi^4|h||\phi|\,dx\,dt+C\lambda^3\mathbb{E}\int_0^T\int_{G_1'} \theta^2\varphi^3|\phi||\nabla h|\,dx\,dt\\
					&\quad+C\lambda^4\mathbb{E}\int_0^T\int_{G_1'} \theta^2\varphi^4|\phi||\nabla h|\,dx\,dt. 
				\end{aligned}
			\end{align}
			Applying Young's inequality in  the right-hand side of \eqref{ine5.16}, we obtain for any $\varepsilon>0$
			\begin{align}\label{inti13}
				\begin{aligned}
					I_3+I_4&\leq\varepsilon\lambda^3\mathbb{E}\int_0^T\int_{G_1'} \zeta\theta^2\varphi^3 h^2 dx\,dt+\frac{C}{\varepsilon}\lambda^3\mathbb{E}\int_{Q_0} \theta^2\varphi^3 \phi^2 dx\,dt\\
					&\quad\,+\frac{C}{\varepsilon}\lambda^7\mathbb{E}\int_{Q_0} \theta^2\varphi^7 \phi^2 dx\,dt+\frac{C}{\varepsilon}\lambda^5\mathbb{E}\int_{Q_0} \theta^2\varphi^5 \phi^2 dx\,dt.\\
					& \quad\,+\varepsilon\lambda\mathbb{E}\int_0^T\int_{G_1'} \theta^2\varphi |\nabla h|^2 dx\,dt.
				\end{aligned}
			\end{align}
			From \eqref{integ11}, \eqref{Int2}, \eqref{Int4} and \eqref{inti13}, we conclude that for a large enough $\lambda$
			\begin{align}\label{Abss1}
				\begin{aligned}
					\lambda^3\mathbb{E}\int_0^T\int_{G_1'} \zeta\theta^2\varphi^3h^2 \,dx\,dt&\leq 2\varepsilon \lambda^3\mathbb{E}\int_0^T\int_{G_1'} \zeta\theta^2\varphi^3h^2 \,dx\,dt+\frac{C}{\varepsilon}\lambda^7\mathbb{E}\int_{Q_0}\theta^2\varphi^7\phi^2 \,dx\,dt\\ 
					&\quad\,+\varepsilon\lambda\mathbb{E}\int_0^T\int_{G_1'} \theta^2\varphi |\nabla h|^2 \,dx\,dt.
				\end{aligned}
			\end{align}
		\end{itemize}
		Therefore from \eqref{Abss1}, we have that for any $\varepsilon>0$
		\begin{align}\label{Abss11}
			\begin{aligned}
				\lambda^3\mathbb{E}\int_0^T\int_{G'} \theta^2\varphi^3h^2 \,dx\,dt&\leq 2\varepsilon \lambda^3\mathbb{E}\int_Q \theta^2\varphi^3h^2 \,dx\,dt+\frac{C}{\varepsilon}\lambda^7\mathbb{E}\int_{Q_0}\theta^2\varphi^7\phi^2 \,dx\,dt\\ 
				&\quad\,+\varepsilon\lambda\mathbb{E}\int_Q \theta^2\varphi |\nabla h|^2 \,dx\,dt.
			\end{aligned}
		\end{align}
		Now, combining \eqref{firsine1} and \eqref{Abss11} and choosing  a small enough $\varepsilon>0$, we end up with
		\begin{align}\label{in5.211}
			\begin{aligned}
				\mathcal{I}(\lambda,\mu;\phi, \phi_{\Gamma}) +\mathcal{I}(\lambda,\mu; h, h_{\Gamma})  &\leq C\Bigg[ \lambda^3 \mathbb{E}\int_{Q_0} \theta^2\varphi^3 \phi^2 \,dx\,dt +\lambda^7 \mathbb{E}\int_{Q_0} \theta^2\varphi^7 \phi^2 \, dx\,dt\\
				&\quad\;\;\quad+\lambda^2\mathbb{E}\int_{Q}\theta^2\varphi^2 \Phi^2\,dx\,dt+\lambda^2\mathbb{E}\int_{\Sigma}\theta^2\varphi\widehat{\Phi}^2\,d\sigma\,dt\Bigg].
			\end{aligned}
		\end{align}
		Finally, by taking a large enough $\lambda$ in \eqref{in5.211}, we deduce the desired Carleman estimate \eqref{Carlem5.9}. This concludes the proof of Lemma \ref{thmm5.1}.
	\end{proof}
	
	To take into account the presence of  the source terms $y_{d,1}$ and $y_{d,2}$ in the optimality system \eqref{eqq4.7}, we need an improved  Carleman inequality. To this end,  we introduce the modified weight functions
	\begin{equation*}\label{eq:adjoint-system} \ell(t)=\begin{cases}
			\begin{array}{ll}
				T^2/4 &\textnormal{if}\;\; t\in(0,T/2),\\
				t(T-t) &\textnormal{if}\;\; t\in (T/2,T),
			\end{array}
		\end{cases}
	\end{equation*}
	and 
	\begin{align}\label{rec1}
		&\,\overline{\alpha}=\overline{\alpha}(t,x) = \frac{e^{\mu\eta(x)}-e^{2\mu\vert\eta\vert_\infty}}{\ell(t)},\qquad \overline{\varphi}=\overline{\varphi}(t,x) =\frac{e^{\mu\eta(x)}}{\ell(t)},\qquad\overline{\theta}=e^{\lambda\overline{\alpha}}.
	\end{align}
	We also denote by
	\begin{equation}\label{rec2}
		\overline{\alpha}^{\star}(t)=\min_{x\in\overline{G}}\overline{\alpha}(t,x),  \quad \text{and} \quad  	\overline{\varphi}^{\star}(t)=\max_{x\in\overline{G}}\overline{\varphi}(t,x).
	\end{equation}
	Set 
	\begin{align*}
		\overline{\mathcal{I}}(\lambda,\mu; \xi, \xi_{\Gamma})& = \mathbb{E}\int_Q\overline{\theta}^2\overline{\varphi}^3 \xi^2\,dx\,dt+\mathbb{E}\int_\Sigma\overline{\theta}^2\overline{\varphi}^3 \xi_\Gamma^2\,d\sigma dt.
	\end{align*}
	
	We now provide the following main improved  Carleman estimate for system \eqref{ADJSO1}.
	\begin{thm}\label{lem4.5st} Let $((\phi,\phi_{\Gamma},\Phi, \widehat{\Phi}), \Psi^1, \Psi^2)$ be the solution of \eqref{ADJSO1} and assume that \eqref{Assump10} holds. Then, for $\mu=\mu_3$ given in Lemma \ref{thmm5.1}, there exist constants $C>0$ and $\lambda_4>1$ depending only on $G$, $G_0$, $\mu_3$ and $T$ such that for all $\lambda\geq\lambda_4$
		\begin{align}\label{improvedCarl}
			\begin{aligned}
				&\,\mathbb{E} |\phi(0)|^2_{L^2(G)}+ \mathbb{E}|\phi_{\Gamma}(0)|^2_{L^2(\Gamma)}+ \overline{\mathcal{I}}(\lambda,\mu; \phi, \phi_{\Gamma})+ \mathcal{I}(\lambda,\mu; h, h_{\Gamma})\\
				&\leq C \Bigg[\lambda^7  \mathbb{E}\int_{Q_0} \theta^2 \varphi^7 \phi^2 \, \,dx\,dt+\lambda^2 \mathbb{E}\int_{Q}\theta^2\varphi^2\Phi^2\,dx\,dt+\lambda^2 \mathbb{E}\int_{\Sigma}\theta^2\varphi\widehat{\Phi}^2 \,d\sigma\,dt\Bigg],
			\end{aligned}
		\end{align}
		where $(h, h_{\Gamma})= \alpha_1\Psi^1+\alpha_2\Psi^2=\alpha_1(\psi^1, \psi^1_{\Gamma})+\alpha_2(\psi^2, \psi^2_{\Gamma})$.
	\end{thm}
	\begin{proof}
		Arguing as in \cite{BoMaOuNash}, let $\kappa\in C^1([0,T])$ such that
		\begin{equation}\label{kappadef}
			\kappa=1\quad \text{in}\,\; [0,T/2],\quad \ \,\,\kappa= 0\quad \text{in} \,\,\;[3T/4, T],\quad \text{and} \,\,\;\;\kappa'\leq C/T^2.
		\end{equation}
		Let $(p, p_{\Gamma})= \kappa (\phi, \phi_{\Gamma})$, $P=\kappa \Phi$, and $\widehat{P}=\kappa\widehat{\Phi}$. Then, $(p, p_{\Gamma}, P, \widehat{P})$ satisfies the system 
		\begin{equation}\label{ADJSO}
			\begin{cases}
				dp+\Delta p \,dt=\big[-a_1 p-a_2 P+\kappa h\chi_{G_d}(x)+ \kappa'\phi\big] \,dt+P \,dW(t) & \text { in }Q, \\ 
				dp_{\Gamma}+\Delta_\Gamma p_\Gamma \,dt-\partial_\nu p \,dt=\big[-b_1p_{\Gamma}-b_2 \widehat{P} + \kappa'\phi_{\Gamma}\big] \,dt+ \widehat{P} \,dW(t) & \text { on }\Sigma, \\ 
				p_{\Gamma}(t, x)=p|_\Gamma(t, x) & \text { on }\Sigma, \\
				(p,p_\Gamma)|_{t=T}=(0,0) & \text { in } G \times \Gamma.
			\end{cases}
		\end{equation}
		Using the estimate \eqref{wbac2.1} and \eqref{kappadef}, we find a positive constant $C$ such that  \begin{align}\label{energy}
			\begin{aligned}
				&\,\mathbb{E}|\phi(0)|^2_{L^2(G)}+\mathbb{E}|\phi_{\Gamma}(0)|^2_{L^2(\Gamma)}+ |(\phi, \phi_{\Gamma})|^2_{L_{\mathcal{F}}^2\left(0, T/2 ; \mathbb{L}^2\right)}\\
				&\leq C\left(\frac{1}{T^2}|(\phi, \phi_{\Gamma})|^2_{L_{\mathcal{F}}^2(T/2, 3T/4 ; \mathbb{L}^2)} +|(h, h_{\Gamma})|^2_{L_{\mathcal{F}}^2(0, 3T/4 ; \mathbb{L}^2)} \right).
			\end{aligned}
		\end{align}
		Since  the weight functions $\overline{\theta}$ and $\overline{\varphi}$ (resp., $\theta$ and $\varphi$ ) are bounded in $(0, T/2)\times\overline{G}$ (resp., $(T/2, 3T/4)\times\overline{G}$), then we have that
		\begin{align*}
			\begin{aligned}
				&\,\mathbb{E}|\phi(0)|^2_{L^2(G)}+\mathbb{E}|\phi_{\Gamma}(0)|^2_{L^2(\Gamma)}+\mathbb{E}\int_{(0,T/2)\times G}\overline{\theta}^2\overline{\varphi}^3\phi^2  \,dx\,dt  +\mathbb{E}\int_{(0,T/2)\times\Gamma}\overline{\theta}^2\overline{\varphi}^3\phi_{\Gamma}^2  \,d\sigma dt\\ 
				&\leq C\left(\frac{1}{T^2}   \mathbb{E}\int_{(T/2,3T/4 )\times G}  \theta^2\varphi^3\phi^2 dx\,dt + \frac{1}{T^2}\mathbb{E}\int_{(T/2,3T/4 )\times\Gamma}  \theta^2\varphi^3\phi_{\Gamma}^2 d\sigma\,dt + |(h, h_{\Gamma})|^2_{L_{\mathcal{F}}^2\left(0, 3T/4 ; \mathbb{L}^2\right)} \right).
			\end{aligned}
		\end{align*}
		Then, it follows that
		\begin{align}\label{energy2}
			\begin{aligned}
				&\,\mathbb{E}|\phi(0)|^2_{L^2(G)}+\mathbb{E}|\phi_{\Gamma}(0)|^2_{L^2(\Gamma)}+\mathbb{E}\int_{(0,T/2)\times G}\overline{\theta}^2\overline{\varphi}^3\phi^2  \,dx\,dt  +\mathbb{E}\int_{(0,T/2)\times\Gamma}\overline{\theta}^2\overline{\varphi}^3\phi_{\Gamma}^2  \,d\sigma dt\\ 
				&\leq C \left(\mathcal{I}(\lambda,\mu;\phi, \phi_{\Gamma}) +  |(h, h_{\Gamma})|^2_{L_{\mathcal{F}}^2\left(0, 3T/4 ; \mathbb{L}^2\right)} \right).
			\end{aligned}
		\end{align}
		Using the fact that  $\theta=\overline{\theta}$ and $\varphi=\overline{\varphi}$ in $(T/2, T)$, we have  \begin{equation}\label{Eq5.18}
			\mathbb{E}\int_{(T/2, T)\times G}\overline{\theta}^2\overline{\varphi}^3\phi^2  \,dx\,dt  +\mathbb{E}\int_{(T/2,T)\times\Gamma}\overline{\theta}^2\overline{\varphi}^3\phi_{\Gamma}^2  \,d\sigma dt
			\leq C \,\mathcal{I}(\lambda,\mu;\phi, \phi_{\Gamma}).
		\end{equation}
		Adding \eqref{energy2} and \eqref{Eq5.18}, we obtain that
		\begin{align}\label{Ineq1}
			\begin{aligned}
				&\,\mathbb{E}|\phi(0)|^2_{L^2(G)}+\mathbb{E}|\phi_{\Gamma}(0)|^2_{L^2(\Gamma)}+\mathbb{E}\int_{Q}\overline{\theta}^2\overline{\varphi}^3\phi^2  \,dx\,dt  +\mathbb{E}\int_{\Sigma}\overline{\theta}^2\overline{\varphi}^3\phi_{\Gamma}^2  \,d\sigma dt\\
				&\leq C \left(\mathcal{I}(\lambda,\mu;\phi,\phi_{\Gamma}) +  |(h, h_{\Gamma})|^2_{L_{\mathcal{F}}^2\left(0, 3T/4 ; \mathbb{L}^2\right)} \right).
			\end{aligned}
		\end{align}
		Using the estimate \eqref{forwenergyes} for the system satisfied by $(h,h_{\Gamma})$ and the fact that the  weight functions $\overline{\theta}$ and $\overline{\varphi}$ are bounded in $[0, 3T/4],$ we see that
		\begin{align}\label{Ineq2}
			|(h,h_{\Gamma})|^2_{L_{\mathcal{F}}^2\left(0, 3T/4 ; \mathbb{L}^2\right)}\leq C\left(\frac{\alpha_1}{\beta_1}+ \frac{\alpha_2}{\beta_2}\right)^2\left[\mathbb{E}\int_{Q}\overline{\theta}^2\overline{\varphi}^3\phi^2  \,dx\,dt  +\mathbb{E}\int_{\Sigma}\overline{\theta}^2\overline{\varphi}^3\phi_{\Gamma}^2  \,d\sigma dt\right].
		\end{align}
		Combining \eqref{Ineq1} and \eqref{Ineq2} and taking a large enough $\beta_1$ and  $\beta_2$,  we deduce that
		\begin{align*} &\,\mathbb{E}|\phi(0)|^2_{L^2(G)}+\mathbb{E}|\phi_{\Gamma}(0)|^2_{L^2(\Gamma)}+\mathbb{E}\int_{Q}\overline{\theta}^2\overline{\varphi}^3\phi^2  \,dx\,dt  +\mathbb{E}\int_{\Sigma}\overline{\theta}^2\overline{\varphi}^3\phi_{\Gamma}^2  \,d\sigma dt\\
			&\leq C \,\mathcal{I}(\lambda,\mu;\phi,\phi_{\Gamma}),
		\end{align*}
		which provides that
		\begin{align}\label{findIneq}
			\begin{aligned}
				&\,\mathbb{E}|\phi(0)|^2_{L^2(G)}+\mathbb{E}|\phi_{\Gamma}(0)|^2_{L^2(\Gamma)}+ \overline{\mathcal{I}}(\lambda,\mu; \phi, \phi_{\Gamma})+ \mathcal{I}(\lambda,\mu; h, h_{\Gamma})\\
				&\leq C \left(\,\mathcal{I}(\lambda,\mu;\phi,\phi_{\Gamma})+ \mathcal{I}(\lambda,\mu; h, h_{\Gamma})\right).
			\end{aligned}
		\end{align}
		We finally combine \eqref{findIneq} and \eqref{Carlem5.9}, we deduce the desired Carleman estimate \eqref{improvedCarl}. This concludes the proof of Theorem \ref{lem4.5st}.
	\end{proof}
	\section{Null controllability}\label{section6} 
	In this section, we prove our main null controllability result given in Theorem \ref{th4.1SN}. To this end, we shall first establish the following key observability inequality.
	\begin{prop}\label{Pro4.2}
		Assuming that \eqref{Assump10} holds  and $\beta_i>0$, $i=1,2$,  are large enough, there exist a constant $C>0$ and a positive weight function $\rho=\rho(t)$ blowing up at $t=T$, such that, for any $(\phi_T,\phi_{\Gamma,T})\in L^2_{\mathcal{F}_T}(\Omega;\mathbb{L}^2)$, the solution $((\phi, \phi_{\Gamma},\Phi, \widehat{\Phi}), \Psi^1, \Psi^2)$ of \eqref{ADJSO1} satisfies the following inequality
		\begin{align}\label{observaineq}
			\begin{aligned}
				&\,\mathbb{E}|\phi(0)|^2_{L^2(G)}+ \mathbb{E}|\phi_{\Gamma}(0)|^2_{L^2(\Gamma)}+ \sum_{i=1}^{2} \left[ \mathbb{E}\int_Q\rho^{-2}|\psi^i|^2\,dx\,dt+\mathbb{E}\int_\Sigma\rho^{-2}|\psi^i_\Gamma|^2\,d\sigma\,dt\right]\\
				&\leq C \left[ \mathbb{E}\int_{Q_0}  \phi^2 \,\,dx\,dt+\mathbb{E}\int_{Q}\Phi^2\,dx\,dt+\mathbb{E}\int_{\Sigma}\widehat{\Phi}^2 \,d\sigma\,dt\right].
			\end{aligned}
		\end{align}
	\end{prop}
	\begin{proof} 
		For fixed $\mu=\mu_3$ and $\lambda=\lambda_4$ (where $\mu_3$ and $\lambda_4$ are given in Theorem \ref{lem4.5st}), we set the weight function $\rho=\rho(t)=e^{-\lambda\overline{\alpha}^{\star}(t)}$.
		By Itô's formula, we have that
		\begin{equation}\label{ito1}
			d(\rho^{-2} (\psi^i)^2)= -2\rho'\rho^{-3} (\psi^i)^2dt+ \rho^{-2} \left[2\psi^i d \psi^i+ (d\psi^i)^2  \right],\qquad i=1,2,
		\end{equation}
		and
		\begin{equation}\label{ito2}
			d( \rho^{-2}(\psi^i_{\Gamma})^2)= -2\rho'\rho^{-3}(\psi^i_{\Gamma})^2dt+ \rho^{-2} \left[2\psi^i_{\Gamma} d \psi^i_{\Gamma}+ (d\psi^i_{\Gamma})^2) \right], \qquad i=1,2.
		\end{equation}
		Using \eqref{ito1} and  \eqref{ito2}, we obtain that the solution $((\phi, \phi_{\Gamma},\Phi, \widehat{\Phi}), \Psi^1, \Psi^2)$ of the system \eqref{ADJSO1} satisfies that for all $t\in(0,T)$
		\begin{align*}
			&\,\mathbb{E}\int_{(0,t)\times G}d(\rho^{-2} (\psi^i)^2)\, dx+\mathbb{E}\int_{(0,t)\times\Gamma}d(\rho^{-2}  (\psi^i_{\Gamma})^2)\, d\sigma\\
			&\leq-2\mathbb{E}\int_{(0,t)\times G}\rho'\rho^{-3} (\psi^i)^2\, dx \,dt-2\mathbb{E}\int_{(0,t)\times\Gamma}\rho'\rho^{-3} (\psi^i_{\Gamma})^2\, d\sigma\,dt\\
			&\quad\;+ 2\mathbb{E}\int_{(0,t)\times G}\rho^{-2}\left[a_1(\psi^i)^2+\frac{1}{\beta_i}\chi_{G_i}(x)\psi^i\phi\right]\, dx \,dt+ \mathbb{E}\int_{(0,t)\times G}\rho^{-2} a_2^2 (\psi^i)^2\, dx \,dt\\
			&\quad\;+2\mathbb{E}\int_{(0,t)\times\Gamma}\rho^{-2}b_1(\psi^i_{\Gamma})^2\, d\sigma \,dt+\mathbb{E}\int_{(0,t)\times\Gamma}\rho^{-2} b_2^2 (\psi^i_{\Gamma})^2\, d\sigma \,dt, \qquad i=1,2.
		\end{align*}
		Then, it follows that
		\begin{align*}
			&\,\mathbb{E}\int_{(0,t)\times G}d(\rho^{-2} (\psi^i)^2)\, dx+\mathbb{E}\int_{(0,t)\times\Gamma}d(\rho^{-2}  (\psi^i_{\Gamma})^2)\, d\sigma\\
			&\leq C\left[\mathbb{E}\int_{(0,t)\times G}\rho^{-2} (\psi^i)^2\, dx \,dt+\mathbb{E}\int_{(0,t)\times\Gamma}\rho^{-2}  (\psi^i_{\Gamma})^2\, d\sigma \,dt+ \mathbb{E}\int_{Q}\rho^{-2} \phi^2\, dx \,dt\right], \qquad i=1,2.
		\end{align*}
		Using the fact that  $\psi^i(0)=\psi^i_{\Gamma}(0)=0$, and the Gronwall inequality, we obtain that
		\begin{align}\label{ineqwithrho}
			\mathbb{E}\int_Q\rho^{-2}|\psi^i|^2\,dx\,dt+\mathbb{E}\int_\Sigma\rho^{-2}|\psi^i_\Gamma|^2\,d\sigma\,dt
			\leq C\,\mathbb{E}\int_{Q}\rho^{-2} \phi^2\, dx \,dt,\qquad i=1,2.
		\end{align}
		Using \eqref{ineqwithrho} and recalling \eqref{rec1} and \eqref{rec2}, we have that
		\begin{align*}
			&\,\mathbb{E}|\phi(0)|^2_{L^2(G)}+\mathbb{E}|\phi_{\Gamma}(0)|^2_{L^2(\Gamma)}+ \sum_{i=1}^{2} \left[ \mathbb{E}\int_Q\rho^{-2}|\psi^i|^2\,dx\,dt+\mathbb{E}\int_\Sigma\rho^{-2}|\psi^i_\Gamma|^2\,d\sigma\,dt\right] \\
			&\leq \mathbb{E}|\phi(0)|^2_{L^2(G)}+\mathbb{E}|\phi_{\Gamma}(0)|^2_{L^2(\Gamma)} +C\,\mathbb{E}\int_{Q}\overline{\theta}^2\overline{\varphi}^3 \phi^2\, dx \,dt.
		\end{align*}
		Then, by Carleman estimate \eqref{improvedCarl}, we deduce that
		\begin{align*}
			&\,\mathbb{E}|\phi(0)|^2_{L^2(G)}+\mathbb{E}|\phi_{\Gamma}(0)|^2_{L^2(\Gamma)}+ \sum_{i=1}^{2} \left[ \mathbb{E}\int_Q\rho^{-2}|\psi^i|^2\,dx\,dt+\mathbb{E}\int_\Sigma\rho^{-2}|\psi^i_\Gamma|^2\,d\sigma\,dt\right] 
			\\&\leq C\left[ \mathbb{E}\int_{Q_0}  \phi^2 \, \,dx\,dt+\mathbb{E}\int_{Q}\Phi^2\,dx\,dt+\mathbb{E}\int_{\Sigma}\widehat{\Phi}^2 \,d\sigma\,dt\right].
		\end{align*}
		This concludes the proof of Proposition \ref{Pro4.2}.
	\end{proof} 
	Now we are in a position to prove the null controllability result of the optimality system \eqref{eqq4.7}.
	\begin{prop}\label{Lm5.5}
		Let  $\rho=\rho(t)$ the  weight function given in Proposition \ref{Pro4.2}. Then, for any  $y_{i,d} \in \mathcal{V}_{i,d}$, $i=1,2,$ satisfying \eqref{inqAss11SN} and  $(y_0,y_{\Gamma, 0} )\in L^2_{\mathcal{F}_0}(\Omega;\mathbb{L}^2)$, there exists a control 
		$$(u_1, u_2, u_3) \in L_{\mathcal{F}}^{2}(0,T; L^2(G_0))\times L_{\mathcal{F}}^{2}(0,T; L^2(G))\times L_{\mathcal{F}}^{2}(0,T; L^2(\Gamma)),$$ 
		with minimal norm such that the corresponding solution of \eqref{eqq4.7} satisfies that
		$$(y(T,\cdot),y_\Gamma(T,\cdot))=(0,0)\;\;\textnormal{in}\;\;G\times\Gamma,\quad\mathbb{P}\textnormal{-a.s.}$$
	\end{prop} 
	\begin{proof}
		Multiplying the solution $((y, y_\Gamma), \mathcal{Z}^1, \mathcal{Z}^2)$ of \eqref{eqq4.7} by the solution $((\phi, \phi_{\Gamma}, \Phi, \widehat{\Phi}), \Psi^1, \Psi^2)$ of \eqref{ADJSO1}, where $\mathcal{Z}^i = (z^i, z_\Gamma^i, Z^i, \widehat{Z}^i)$ and $\Psi^i = (\psi^i, \psi_\Gamma^i)$, $i=1,2$, and then taking the expectation on both sides followed by integration by parts, we find that
		\begin{align}\label{dual}
			\begin{aligned}
				&\,\mathbb{E}\left\langle(y(T),y_{\Gamma}(T)),(\phi_T,\phi_{\Gamma,T})\right\rangle_{\mathbb{L}^2}-\mathbb{E}\left\langle(y_0,y_{\Gamma,0}),(\phi(0),\phi_{\Gamma}(0))\right\rangle_{\mathbb{L}^2} \\
				&=\mathbb{E}\int_{Q_0} u_1\phi \, dx \,dt +\mathbb{E}\int_{Q} u_2\Phi \, dx \,dt +\mathbb{E}\int_{\Sigma} u_3\widehat{\Phi} \, d\sigma \,dt\\ &\quad\,+\sum_{i=1}^{2}\alpha_i \mathbb{E}\int_{ (0,T)\times G_{i,d}} y_{i,d} \psi^i dx \,dt.
			\end{aligned}
		\end{align}
		Thus,  the null controllability property is equivalent to finding, for each $(y_0,y_{\Gamma, 0} )\in L^2_{\mathcal{F}_0}(\Omega;\mathbb{L}^2)$, a control $(u_1, u_2, u_3)$ such that, for any $(\phi_T,\phi_{\Gamma,T})\in L^2_{\mathcal{F}_T}(\Omega;\mathbb{L}^2)$, one has
		\begin{align*}
			&\,\mathbb{E}\left\langle(y_0,y_{\Gamma,0}),(\phi(0),\phi_{\Gamma}(0))\right\rangle_{\mathbb{L}^2} +\mathbb{E}\int_{Q_0} u_1\phi \, dx \,dt +\mathbb{E}\int_{Q} u_2\Phi \, dx \,dt\\
			&+\mathbb{E}\int_{\Sigma} u_3\widehat{\Phi} \, d\sigma \,dt+\sum_{i=1}^{2}\alpha_i \mathbb{E}\int_{ (0,T)\times G_{i,d}} y_{i,d} \psi^i dx \,dt=0.
		\end{align*}
		To this end, let $\varepsilon>0$ and  $(\phi_T,\phi_{\Gamma,T})\in L^2_{\mathcal{F}_T}(\Omega;\mathbb{L}^2)$. Introduce the following functional
		\begin{align*}
			J_{\varepsilon}(\phi_T,\phi_{\Gamma,T})&= \frac{1}{2}\mathbb{E}\int_{Q_0}\phi^2 \,dx\,dt +\frac{1}{2}\mathbb{E}\int_{Q}\Phi^2 \,dx\,dt +\frac{1}{2}\mathbb{E}\int_{\Sigma}\widehat{\Phi}^2 \,d\sigma\,dt +\varepsilon\mathbb{E}|(\phi_T,\phi_{\Gamma,T})|_{\mathbb{L}^2} \\
			&\quad+\mathbb{E}\left\langle(y_0,y_{\Gamma,0}),(\phi(0),\phi_{\Gamma}(0))\right\rangle_{\mathbb{L}^2}+ \sum_{i=1}^{2}\alpha_i \mathbb{E}\int_{(0,T)\times G_{i,d}} y_{i,d} \psi^i \,dx \,dt.
		\end{align*}
		It is easy to see that $J_{\varepsilon}: L^2_{\mathcal{F}_T}(\Omega;\mathbb{L}^2)\longrightarrow\mathbb{R}$ is continuous and strictly convex. Moreover, from Young inequality together with the observability inequality \eqref{observaineq}, we have that for any $\delta>0$
		\begin{align*}
			\mathbb{E}\left\langle(y_0,y_{\Gamma,0}),(\phi(0),\phi_{\Gamma}(0))\right\rangle_{\mathbb{L}^2}&\geq -\frac{1}{2\delta}\mathbb{E}|(y_0,y_{\Gamma,0})|^2_{\mathbb{L}^2} - \frac{\delta}{2}\mathbb{E}|(\phi(0),\phi_{\Gamma}(0))|^2_{\mathbb{L}^2}\\
			&\geq  - \frac{\delta}{2}C\left( \mathbb{E}\int_{Q_0}  \phi^2 \, \,dx\,dt+\mathbb{E}\int_{Q}\Phi^2\,dx\,dt+\mathbb{E}\int_{\Sigma}\widehat{\Phi}^2 \,d\sigma\,dt\right) \\
			&\quad\,+ \frac{\delta}{2}\,\sum_{i=1}^2\left(\mathbb{E}\int_{Q}\rho^{-2}|\psi^i|^2  \,dx\,dt +\mathbb{E}\int_{\Sigma}\rho^{-2}|\psi^i_\Gamma|^2 \,d\sigma\,dt\right)\\
			&\quad\,-\frac{1}{2\delta }\mathbb{E}|(y_0,y_{\Gamma,0})|^2_{\mathbb{L}^2}.\\
		\end{align*}
		Then, it follows that
		\begin{align}\label{inneq1}
			\begin{aligned}
				\mathbb{E}\left\langle(y_0,y_{\Gamma,0}),(\phi(0),\phi_{\Gamma}(0))\right\rangle_{\mathbb{L}^2}&\geq  - \frac{\delta}{2}C\left( \mathbb{E}\int_{Q_0}  \phi^2 \, \,dx\,dt+\mathbb{E}\int_{Q}\Phi^2\,dx\,dt+\mathbb{E}\int_{\Sigma}\widehat{\Phi}^2 \,d\sigma\,dt\right)\\
				&\quad\,+ \frac{\delta}{2}\,\sum_{i=1}^2\mathbb{E}\int_{(0,T)\times G_{i,d}}\rho^{-2}|\psi^i|^2  \,dx\,dt-\frac{1}{2\delta }\mathbb{E}|(y_0,y_{\Gamma,0})|^2_{\mathbb{L}^2}.
			\end{aligned}
		\end{align}
		On the other hand, we have that
		\begin{align}\label{inneq2}
			\begin{aligned}
				\sum_{i=1}^{2}\alpha_i \mathbb{E}\int_{(0,T)\times G_{i,d}} y_{i,d} \psi^i \,dx \,dt&\geq -\frac{1}{2\delta}\sum_{i=1}^{2}\alpha_i^2 \mathbb{E}\int_{(0,T)\times G_{i,d}} \rho^2|y_{i,d}|^2 \,dx \,dt\\
				&\quad\,-  \frac{\delta}{2}\sum_{i=1}^{2} \mathbb{E}\int_{(0,T)\times G_{i,d}} \rho^{-2} |\psi^i|^2 \,dx \,dt.
			\end{aligned}
		\end{align}
		Combining \eqref{inneq1} and \eqref{inneq2} and choosing $\delta=\frac{1}{2C}$ (with $C$ is the same constant as in \eqref{inneq1}), we get
		\begin{align*}
			J_{\varepsilon}(\phi_T,\phi_{\Gamma,T})\geq&\,\frac{1}{4}\left( \mathbb{E}\int_{Q_0}  \phi^2 \, \,dx\,dt+\mathbb{E}\int_{Q}\Phi^2\,dx\,dt+\mathbb{E}\int_{\Sigma}\widehat{\Phi}^2 \,d\sigma\,dt\right) +\varepsilon\mathbb{E}|(\phi_T,\phi_{\Gamma,T})|_{\mathbb{L}^2} \\
			&- C\left( \mathbb{E}|(y_0,y_{\Gamma,0})|^2_{\mathbb{L}^2}  
			+ \sum_{i=1}^{2}\alpha^2_i \mathbb{E}\int_{(0,T)\times G_{i,d}} \rho^2|y_{i,d}|^2 dx \,dt\right).
		\end{align*}
		Consequently, $J_{\varepsilon}$ is coercive in $L^2_{\mathcal{F}_T}(\Omega;\mathbb{L}^2)$, and thus $J_{\varepsilon}$ admits a unique minimum $(\phi^\varepsilon_T,\phi^\varepsilon_{\Gamma,T})$. If $(\phi^\varepsilon_T,\phi^\varepsilon_{\Gamma,T})\neq(0,0)$, we  have that
		\begin{equation}
			\langle J'(\phi^\varepsilon_T,\phi^\varepsilon_{\Gamma,T}), (\phi_T,\phi_{\Gamma,T})\rangle_{L^2_{\mathcal{F}_T}(\Omega;\mathbb{L}^2)} =0,\quad \,\,\forall(\phi_T,\phi_{\Gamma,T})\in L^2_{\mathcal{F}_T}(\Omega;\mathbb{L}^2).
		\end{equation}
		Then, for all $(\phi_T,\phi_{\Gamma,T})\in L^2_{\mathcal{F}_T}(\Omega;\mathbb{L}^2)$, we obtain that
		\begin{align}\label{Eq51NS}
			\begin{aligned}
				&\,\mathbb{E}\int_{Q_0}\phi^{\varepsilon}\,\phi \,dx\,dt +\mathbb{E}\int_{Q}\Phi^{\varepsilon}\,\Phi \,dx\,dt+\mathbb{E}\int_{\Sigma}\widehat{\Phi^{\varepsilon}}\,\widehat{\Phi }\,d\sigma\, dt+\varepsilon\mathbb{E}\left\langle\frac{(\phi^\varepsilon_T,\phi^\varepsilon_{\Gamma,T})}{|(\phi^\varepsilon_T,\phi^\varepsilon_{\Gamma,T})|_{L^2_{\mathcal{F}_T}(\Omega;\mathbb{L}^2)}},(\phi_T,\phi_{\Gamma,T})\right\rangle_{\mathbb{L}^2}\\
				&+\mathbb{E}\left\langle(y_0,y_{\Gamma,0}),(\phi(0),\phi_{\Gamma}(0))\right\rangle_{\mathbb{L}^2}+ \sum_{i=1}^{2}\alpha_i \mathbb{E}\int_{{(0,T)\times G_{i,d}}} y_{i,d} \psi^i dx \,dt=0.
			\end{aligned}
		\end{align}
		Taking controls $(u^{\varepsilon}_{1},u^{\varepsilon}_{2},u^{\varepsilon}_{3})=(\phi_{\varepsilon},\Phi_{\varepsilon},\widehat{\Phi_{\varepsilon}})$ in \eqref{dual}, and combining the resulting equality with \eqref{Eq51NS}, we find that
		\begin{equation*}
			\varepsilon\mathbb{E}\left\langle\frac{(\phi^\varepsilon_T,\phi^\varepsilon_{\Gamma,T})}{|(\phi^\varepsilon_T,\phi^\varepsilon_{\Gamma,T})|_{L^2_{\mathcal{F}_T}(\Omega;\mathbb{L}^2)}},(\phi_T,\phi_{\Gamma,T})\right\rangle_{\mathbb{L}^2}+\mathbb{E}\left\langle(y_\varepsilon(T),y_{\Gamma,\varepsilon}(T)),(\phi_T,\phi_{\Gamma,T})\right\rangle_{\mathbb{L}^2}=0,
		\end{equation*}
		for all $(\phi_T,\phi_{\Gamma,T})\in L^2_{\mathcal{F}_T}(\Omega;\mathbb{L}^2)$. Hence 
		\begin{equation}\label{eq53NS}
			\mathbb{E}\left|(y_\varepsilon(T),y_{\Gamma,\varepsilon}(T))\right|_{\mathbb{L}^2}\leq \varepsilon.
		\end{equation}
		Choosing $(\phi,\Phi,\widehat{\Phi})=(\phi_{\varepsilon},\Phi_{\varepsilon},\widehat{\Phi_{\varepsilon}})$ in \eqref{Eq51NS} and using the observability inequality \eqref{observaineq} together with  Young inequality, we deduce that
		\begin{align}\label{eq54NS}
			\begin{aligned}
				&\,|u^\varepsilon_{1}|^2_{ L^2_\mathcal{F}(0,T;L^2(G_0))}+|u^\varepsilon_{2}|^2_{L^2_\mathcal{F}(0,T;L^2(G))}+|u^\varepsilon_{3}|^2_{L^2_\mathcal{F}(0,T;L^2(\Gamma))}\\
				&\leq C\left(\mathbb{E}|(y_0,y_{\Gamma,0})|^2_{\mathbb{L}^2}+ \sum_{i=1}^{2}\alpha^2_i \mathbb{E}\int_{(0,T)\times G_{i,d}} \rho^2 y^2_{i,d} \,dx \,dt\right).
			\end{aligned}
		\end{align}
		If $(\phi^\varepsilon_T,\phi^\varepsilon_{\Gamma,T})=(0,0)$, we obtain that
		\begin{equation}\label{Ezu.1}
			\lim\limits_{t\rightarrow 0^{+}}\frac{J_{\varepsilon}(t \,(\phi_T,\phi_{\Gamma,T}))}{t}\geq0, \qquad \forall(\phi_T,\phi_{\Gamma,T}) \in L^2_{\mathcal{F}_T}(\Omega;\mathbb{L}^2).
		\end{equation}
		Using \eqref{Ezu.1} and take $(u^{\varepsilon}_1,u^{\varepsilon}_2,u^{\varepsilon}_3) = (0,0,0)$, the inequalities \eqref{eq53NS} and \eqref{eq54NS} hold. 
		
		By \eqref{eq54NS}, there exist a subsequence (denoted also by $(u^{\varepsilon}_1,u^{\varepsilon}_2, u^{\varepsilon}_3)$) of $(u^{\varepsilon}_1,u^{\varepsilon}_2, u^{\varepsilon}_3)$ such that as $\varepsilon\rightarrow0$
		\begin{align}\label{weakconvr}
			\begin{aligned}
				u^{\varepsilon}_1\longrightarrow  u_1\quad  \text{weakly in} \,\,\; L^2((0,T)\times\Omega;L^2(G_0));&\\
				u^{\varepsilon}_2\longrightarrow  u_2\quad  \text{weakly in} \,\,\; L^2((0,T)\times\Omega;L^2(G));&\\
				u^{\varepsilon}_3\longrightarrow  u_3\quad  \text{weakly in} \,\,\; L^2((0,T)\times\Omega;L^2(\Gamma)).
			\end{aligned}
		\end{align}
		By \eqref{weakconvr} and energy estimate, it is easy to see that
		\begin{equation}\label{Eq56}
			(y_\varepsilon(T),y_{\Gamma,\varepsilon}(T))\longrightarrow  (y(T),y_{\Gamma}(T))\quad  \text{weakly in} \;\,\,  L^2_{\mathcal{F}_T}(\Omega;\mathbb{L}^2),\quad \textnormal{as}\; \varepsilon\rightarrow0.
		\end{equation}
		Combining \eqref{eq53NS} and \eqref{Eq56}, we finally deduce that $$(y(T,\cdot),y_\Gamma(T,\cdot))=(0,0)\;\;\textnormal{in}\;\;G\times\Gamma,\quad\mathbb{P}\textnormal{-a.s.}$$
		This concludes the proof of Proposition \ref{Lm5.5} and then establishes our null controllability result as outlined in Theorem \ref{th4.1SN}.
	\end{proof} 
	\section{Conclusion and discussion}\label{section7}
	In this paper, we applied the Stackelberg and Nash strategies to a class of forward stochastic parabolic equations with dynamic boundary conditions. To solve the problem, we combined game theory concepts with some well-known controllability techniques. After characterizing the Nash equilibrium, the original problem was transformed into the null controllability issue for a coupled stochastic forward-backward system, which was solved using the Carleman estimates approach.
	
	Based on the results obtained in this paper, some points are in order:
	\begin{itemize}
		\item To simplify the exposition of our work, we have used the same white noise on both the bulk and boundary equations. However, the results also hold with different noises $W^1(\cdot)$ within the domain $G$ and $W^2(\cdot)$ along the boundary $\Gamma$, where $(W^1(\cdot),W^2(\cdot))$ is a two-dimensional Brownian motion. 
		\item The introduction of three leaders represents a technical constraint. This is a classical issue when dealing with the controllability of forward stochastic parabolic equations. However, we can handle the backward system with only one leader $u$ and two followers $v_1$ and $v_2$. More precisely, using the same ideas as presented in this paper, we can treat Stackelberg-Nash null controllability of the following system
		\begin{equation*}
			{\qquad\qquad\begin{cases}
					\begin{array}{ll}
						dy +\Delta y \,dt = \left[a_1y+a_2Y+\chi_{G_0}(x)u+\chi_{G_1}(x)v_1+\chi_{G_2}(x)v_2\right] \,dt + Y\,dW(t)&\textnormal{in}\,\,Q,\\
						dy_\Gamma+\Delta_\Gamma y_\Gamma \,dt-\partial_\nu y \,dt = \left[b_1y_\Gamma+b_2\widehat{Y}\right]\,dt+\widehat{Y} \,dW(t) &\textnormal{on}\,\,\Sigma,\\
						y_\Gamma(t,x)=y\vert_\Gamma(t,x) &\textnormal{on}\,\,\Sigma,\\
						(y,y_\Gamma)\vert_{t=T}=(y_T,y_{\Gamma,T}) &\textnormal{in}\,\,G\times\Gamma.
					\end{array}
			\end{cases}}
		\end{equation*}
		\item In the case of the forward equation \eqref{eqq1.1} with spatially independent coefficients, we can employ the spectral method to derive the desired observability inequality \eqref{observaineq} for the adjoint equation. References \cite{elgmanforwione,lu2011some,observineqback} offer insights into the classical null controllability result. In this case, it is interesting to investigate the Stackelberg-Nash null controllability of \eqref{eqq1.1} with only one leader.
		\item Studying the Stackelberg-Nash controllability of general classes of forward stochastic parabolic equations, incorporating dynamic boundary conditions alongside first-order and variable-coefficient second-order terms, as well as the semilinear case, constitutes two important questions. We refer to \cite{Preprintelgrou23,BackSPEwithDBC} for the controllability of some general stochastic parabolic equations. See also \cite{san23,ZhangGuo} for the controllability results of a class of semilinear stochastic parabolic equations.
		\item It would be quite interesting to study the Stackelberg-Nash controllability  of equation \eqref{eqq1.1} when we reverse the roles of the leaders and the followers. For established results in the deterministic case, we refer to \cite{BoMaOuNash2}.
		
	\end{itemize}

\end{document}